\documentclass[11pt, reqno]{amsart}

\usepackage[left=2.5cm, right=2.5cm, top=2cm, bottom=2cm]{geometry}

\usepackage[utf8]{inputenc}

\usepackage{amssymb, amsmath, amsthm, esint, enumitem,dsfont,tikz}
\usetikzlibrary{angles,quotes}

\usepackage{latexsym}

\usepackage{graphicx}

\usepackage{hyperref}

\usepackage{mathtools}
\mathtoolsset{showonlyrefs}

\usepackage{color}

\newtheorem{defi}{Definition}[section]
\newtheorem{thm}[defi]{Theorem}

\newtheorem{rem}[defi]{Remark}
\newtheorem{prop}[defi]{Proposition}
\newtheorem{lemma}[defi]{Lemma}

\newcommand{\R}{\mathbb R}

\title[Heat flow in polygons with reflecting edges]{Heat flow in polygons with reflecting edges}
\author{Sam Farrington}
\address{Department of Mathematical Sciences, Durham University,
Mathematical Sciences and Computer Science Building MCS2030,
Upper Mountjoy Campus, Stockton Road,
Durham DH1 3LE,
United Kingdom.}
\email{sam.farrington@durham.ac.uk}
\author{Katie Gittins}
\address{Department of Mathematical Sciences, Durham University,
Mathematical Sciences and Computer Science Building MCS2030,
Upper Mountjoy Campus, Stockton Road,
Durham DH1 3LE,
United Kingdom.}
\email{katie.gittins@durham.ac.uk}
\date{\today}

\begin{document}
\begin{abstract}
We investigate the heat flow in an open, bounded set $D$ in $\R^2$ with polygonal boundary $\partial D$.
We suppose that $D$ contains an open, bounded set $\widetilde{D}$ with polygonal boundary $\partial \widetilde{D}$.
The initial condition is the indicator function of $\widetilde{D}$
and we impose a Neumann boundary condition on the edges of $\partial D$.
We obtain an asymptotic formula for the heat content of $\widetilde{D}$ in $D$ as time $t\downarrow 0$.
\end{abstract}
\maketitle

\section{Introduction}

Let $D\subset \mathbb{R}^{2}$ be an open, bounded, connected set with polygonal boundary $\partial D$ such that the interior angles $\theta$ of $\partial D$ satisfy $0 < \theta < 2\pi$. We call such a set $D$ a polygonal domain.
Let $\widetilde{D}\subset D$ be a polygonal subdomain.
We consider the heat equation \begin{equation}
    \frac{\partial u}{\partial t} = \Delta u
\end{equation} on $D$ with Neumann boundary condition imposed on $\partial D$, that is
\begin{equation}
    \frac{\partial u}{\partial n} (t;x) = 0, \enspace x \in \partial D, \, t>0,
\end{equation}
where $n$ is the inward-pointing unit normal to $\partial D$ (defined up to a set of measure zero), and initial datum \begin{equation}
    \lim_{t\downarrow 0} u(t;x) = \mathds{1}_{\widetilde{D}}(x), \enspace x \in D\setminus \partial \widetilde{D}.
\end{equation}
Here $\mathds{1}_{\widetilde{D}}$ is the indicator function of $\widetilde{D}$.
We denote the unique, smooth solution to this problem by $u_{D,\widetilde{D}}$. Physically, $u_{D,\widetilde{D}}(t;x)$ represents the temperature at $x\in D$ at time $t$ when the initial temperature is 1 in $\widetilde{D}$ and $0$ in $D\setminus\widetilde{D}$ and the total heat contained in $D$ is conserved due to the Neumann (fluxless or insulated) boundary condition. The solution $u_{D,\widetilde{D}}$ can be obtained from the unique Neumann heat kernel $\eta_{D}(t;x,y)$ for $D$ by
\begin{equation}
    u_{D,\widetilde{D}}(t;x) = \int_{D} dy \;  \eta_{D}(t;x,y) \mathds{1}_{\widetilde{D}}(y) = \int_{\widetilde{D}} dy \; \eta_{D}(t;x,y).
\end{equation} We are interested in the heat content of $\widetilde{D}$, that is the amount of heat inside $\widetilde{D}$ at time $t$, and we denote it by \begin{equation}\label{eq:1}
    H_{D,\widetilde{D}}(t) := \int_{\widetilde{D}} dx \; u_{D,\widetilde{D}}(t;x).
\end{equation}

The function $H_{D,\widetilde{D}}(t)$ is smooth from the smoothness of $u_{D,\widetilde{D}}$ and other properties can be deduced from the $L^2$-eigenfunction expansion of $\eta_{D}$. Let $\mu_j$, $j=0,1,\dots$, denote the Neumann eigenvalues of $D$ with corresponding eigenfunctions $\phi_j$ that are normalised in $L^2(D)$. We recall that the Neumann eigenvalues can be written in a non-decreasing sequence counted with multiplicity $0=\mu_0 < \mu_1 \leq \mu_2 \leq \dots$, where the only accumulation point is $+\infty$.
The Neumann heat kernel has the following expansion
\begin{equation*}
    \eta_{D}(t;x,y) = \sum_{j=0}^\infty e^{-t \mu_j} \phi_j(x) \phi_j(y).
\end{equation*}
The heat content of $\widetilde{D}$ can then be written
\begin{equation*}
    H_{D,\widetilde{D}}(t)
    = \sum_{j=0}^\infty e^{-t \mu_j} \left(\int_{\widetilde{D}} dx \, \phi_j(x) \right)^2,
\end{equation*}
from which we observe that $t \mapsto H_{D,\widetilde{D}}(t)$
is strictly decreasing on $(0,\infty)$,
\begin{equation*}
    \frac{\partial}{\partial t} H_{D,\widetilde{D}}(t)
    = \sum_{j=0}^\infty -\mu_j e^{-t \mu_j} \left(\int_{\widetilde{D}} dx \, \phi_j(x) \right)^2 < 0,
\end{equation*}
and strictly convex on $(0,\infty)$,
\begin{equation*}
    \frac{\partial^2}{\partial t^2} H_{D,\widetilde{D}}(t)
    = \sum_{j=0}^\infty \mu_j^2 e^{-t \mu_j} \left(\int_{\widetilde{D}} dx \, \phi_j(x) \right)^2 > 0.
\end{equation*}
Since $\phi_0(x) = \vert D \vert^{-1/2}$, we see that
\begin{equation*}
    \lim_{t\to \infty} H_{D,\widetilde{D}}(t)
    = \lim_{t\to \infty} \left( \left(\int_{\widetilde{D}} dx \, \phi_0(x) \right)^2 + \sum_{j=1}^\infty e^{-t \mu_j} \left(\int_{\widetilde{D}} dx \, \phi_j(x) \right)^2\right)
    = \frac{\vert \widetilde{D} \vert^2}{\vert D \vert}.
\end{equation*}

In the rest of this paper, we consider the small-time asymptotic behaviour of $H_{D,\widetilde{D}}(t)$. Thus, in what follows, by $O(\cdot)$ we consider the limit $t \downarrow 0$.

The heat content of polygonal subdomains in, possibly unbounded, domains $\Omega \subset \mathbb{R}^{2}$ can be defined analogously to \eqref{eq:1}
by considering the heat equation on $\Omega$ with some boundary condition imposed on $\partial \Omega$ (when the latter is non-empty).
The small-time asymptotics for such cases have been obtained in \cite{vdBS90}, \cite{vdBG16}, and \cite{vdBGG20} and we summarise these below.
We denote the length of a segment $A \subset \partial \widetilde{D}$ by $L(A)$ so that $L(\partial \widetilde{D})$ is the length of the boundary of $\widetilde{D}$.

In \cite{vdBS90}, the authors consider the Dirichlet case, that is $\Omega=\widetilde{D}$, $\widetilde{D}$ has initial temperature 1 and Dirichlet boundary condition imposed on $\partial \widetilde{D}$ for all $t>0$. The Dirichlet heat content of $\widetilde{D}$ is defined as
\begin{equation*}
    Q_{\widetilde{D}}(t) := \int_{\widetilde{D}} dx \int_{\widetilde{D}} dy \, \pi_{\widetilde{D}}(t; x,y),
\end{equation*}
where $\pi_{\widetilde{D}}(t;x,y)$ is the Dirichlet heat kernel of $\widetilde{D}$, and the authors obtain that
\begin{equation}
    Q_{\widetilde{D}}(t)  = \vert \widetilde{D} \vert - \frac{2 L(\partial \widetilde{D})}{\pi^{1/2}} t^{1/2} + \left(\sum_{\gamma \in \mathcal{A}} f(\gamma)\right)t + O\left(e^{-C_{1}/t}\right),
\end{equation} where: $\mathcal{A}$ is the collection of interior angles at the vertices of $\widetilde{D}$; $C_{1} > 0$ is a constant depending only on $\widetilde{D}$; and, $f:(0,2\pi) \to \mathbb{R}$ is given by \begin{equation}
    f(\gamma) := \int_{0}^{\infty}d\theta \; \frac{4\sinh\left(\left(\pi-\gamma\right)\theta\right)}{\sinh\left(\pi \theta\right) \cosh\left(\gamma \theta\right)}.
\end{equation}
We note that in \cite{vdBS90}, the authors include the case where interior angles can be equal to $2\pi$. One can also introduce this for Neumann boundary conditions under a suitable generalisation of the boundary condition which we discuss in Appendix B.

In \cite{vdBG16}, the authors consider the open case, that is $\Omega = \mathbb{R}^{2}$, $\widetilde{D}$ has initial temperature 1 and $\mathbb{R}^2 \setminus \overline{\widetilde{D}}$ has initial temperature 0. The open heat content of $\widetilde{D}$ is defined as
\begin{equation*}
    H_{\widetilde{D}}(t) := \int_{\widetilde{D}} dx \int_{\widetilde{D}} dy \, p_{\mathbb{R}^2}(t;x,y),
\end{equation*}
where $p_{\mathbb{R}^2}(t;x,y) = (4\pi t)^{-1} e^{-\vert x-y \vert^2/(4t)}$ is the heat kernel for $\mathbb{R}^2$. The authors obtain that
\begin{equation}
    H_{\widetilde{D}}(t) = \vert \widetilde{D}\vert - \frac{L(\partial \widetilde{D})}{\pi^{1/2}} t^{1/2} +  \left(\sum_{\gamma \in \mathcal{A}} a(\gamma)\right)t + O\left(e^{-C_{2}/t}\right),
\end{equation} where: $\mathcal{A}$ is as above; $C_{2} > 0$ is a constant depending only on $\widetilde{D}$; and, $a:(0,2\pi) \to \mathbb{R}$ is given by \begin{equation}\label{eq:e1}
    a(\gamma) := \begin{cases}
 \frac{1}{\pi}+\left(1-\frac{\gamma}{\pi}\right)\cot\gamma , & \gamma \in (0,\pi)\cup(\pi,2\pi), \\
 0, & \gamma =\pi.
    \end{cases}
\end{equation}

In \cite{vdBGG20}, the authors consider the Dirichlet-open case, that is $\Omega = \mathbb{R}^{2}\setminus \partial\widetilde{D}_{-}$, where $\partial \widetilde{D}_{-} \subset \partial\widetilde{D}$ is some collection of edges, $\widetilde{D}$ has initial temperature 1, $\mathbb{R}^2 \setminus \overline{\widetilde{D}}$ has initial temperature 0 and Dirichlet boundary condition imposed on $\partial \widetilde{D}_{-}$ for all $t>0$.
The corresponding heat content of $\widetilde{D}$ is defined as
\begin{equation*}
    G_{\widetilde{D}, \partial \widetilde{D}_{-}}(t) := \int_{\widetilde{D}} dx \int_{\widetilde{D}} dy \, \pi_{\mathbb{R}^2 \setminus \partial \widetilde{D}_{-}}(t;x,y).
\end{equation*}
The authors obtain
\begin{equation}
\begin{split}
     G_{\widetilde{D}, \partial \widetilde{D}_{-}}(t)
     & = |\widetilde{D}| - \frac{\left(L(\partial \widetilde{D} \setminus \partial \widetilde{D}_{-}) +2L(\partial \widetilde{D}_{-})\right)}{\pi^{1/2}}t^{1/2} \\
   & \qquad + \left(\sum_{\gamma\in \mathcal{A}_{1}}a(\gamma) + \sum_{\gamma\in \mathcal{A}_{2}}f(\gamma)+\sum_{\gamma\in \mathcal{A}_{3}}g(\gamma)\right)t + O\left(e^{-C_{3} / t}\right),
 \end{split}
\end{equation} where: $\mathcal{A}_{1}$ is the collection of interior angles at vertices where two open edges intersect; $\mathcal{A}_{2}$ is the collection of interior angles at vertices where two Dirichlet edges intersect; $\mathcal{A}_{3}$ is the collection of interior angles at vertices where a Dirichlet edge and an open edge intersect; $C_{3}>0$ is a constant only depending on $\widetilde{D}$; $a$ and $f$ are as above; and $g : (0,2\pi) \to \mathbb{R}$ is given by \begin{equation}
    g(\gamma) := -\frac{3}{4}+\int_{0}^{\infty}d\theta \; \frac{4\sinh^{2}\left(\left(\pi-\tfrac{\gamma}{2}\right)\theta\right)-\sinh^{2}\left(\left(\pi-\gamma \right)\theta\right)}{\sinh^{2}\left(\tfrac{\pi}{2}\theta\right)\cosh\left(\pi\theta\right)}.
\end{equation}

In the Dirichlet, open, and Dirichlet-open cases any boundary conditions are only imposed on a subset of $\partial \widetilde{D}$ and so the results only ever depend on the geometry of $\widetilde{D}$. For the problem we consider in this paper, this is not the case and we also have the relative geometry of $\widetilde{D}$ with respect to $D$ to consider.
Thus, in order to state our main result, we require some additional terminology and notation.
Let $\widetilde{V}$ denote the union of the vertices of $\widetilde{D}$ and the vertices of $D$ lying on $\partial \widetilde{D}$. Moreover, let $\widetilde{E}$ denote the collection of edges of parts of the boundary of $\partial \widetilde{D}$ between vertices in $\widetilde{V}$.
We call edges in $\widetilde{E}$ that lie in $D$ (except, perhaps, the endpoints of the edges) open edges and edges in $\widetilde{E}$ that lie on $\partial D$ Neumann edges.
Throughout this work we assume each vertex in $\widetilde{V}$ is of one of the following types (see Figure \ref{fig:example_polygon} for an example).

 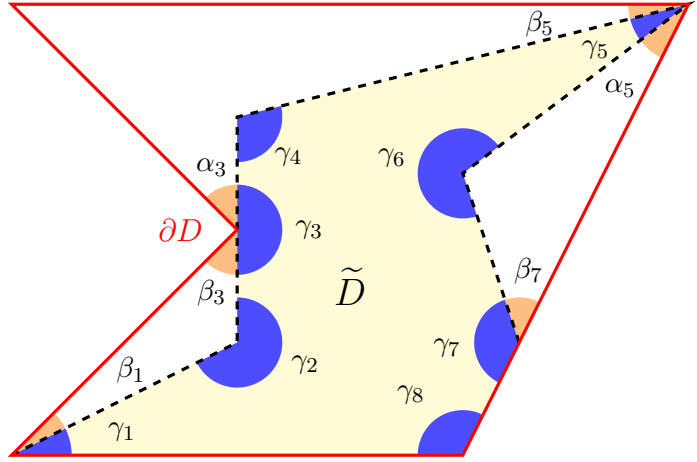
\begin{figure}
    \centering
    \begin{tikzpicture}[scale=1.5]
    \path [fill=yellow!20] (0,0) -- (2,2) -- (0,4) -- (6,4) -- (4,0) -- cycle;
    \path [fill=white] (0,0) -- (2,1) -- (2,2) -- cycle;
    \path [fill=white] (2,2) -- (2,3) -- (6,4) -- (0,4) -- cycle;
    \path [fill=white] (6,4) -- (4,2.5) -- (4.5,1) -- cycle;
     \coordinate (a) at (0,0);
    \coordinate (b) at (2,1);
    \coordinate (c) at (2,2);
    \coordinate (d) at (2,3);
    \coordinate (e) at (6,4);
    \coordinate (f) at (4,2.5);
    \coordinate (g) at (4.5,1);
    \coordinate (h) at (4,0);
    \coordinate (i) at (0,4);
    \draw pic[fill=blue!70,angle radius=0.8cm,"$\gamma_{1}$" shift={(10mm,2mm)}] {angle=h--a--b};
    \draw pic[fill=blue!70,angle radius=0.6cm,"$\gamma_{2}$" shift={(6mm,-1mm)}] {angle=a--b--c};
    \draw pic[fill=blue!70,angle radius=0.6cm,"$\gamma_{3}$" shift={(6mm,0mm)}] {angle=b--c--d};
    \draw pic[fill=blue!70,angle radius=0.6cm,"$\gamma_{4}$" shift={(4mm,-3mm)}] {angle=c--d--e};
    \draw pic[fill=blue!70,angle radius=0.8cm,"$\gamma_{5}$" shift={(-8mm,-4mm)}] {angle=d--e--f};
    \draw pic[fill=blue!70,angle radius=0.6cm,"$\gamma_{6}$" shift={(-6mm,1mm)}] {angle=e--f--g};
    \draw pic[fill=blue!70,angle radius=0.6cm,"$\gamma_{7}$" shift={(-6mm,-1mm)}] {angle=f--g--h};
    \draw pic[fill=blue!70,angle radius=0.6cm,"$\gamma_{8}$" shift={(-5mm,5mm)}] {angle=g--h--a};
    \draw pic[fill=orange!50,angle radius=0.8cm,"$\beta_{1}$" shift={(12mm,8.5mm)}] {angle=b--a--c};
    \draw pic[fill=orange!50,angle radius=0.6cm,"$\beta_{3}$" shift={(-2mm,-5mm)}] {angle=a--c--b};
    \draw pic[fill=orange!50,angle radius=0.6cm,"$\alpha_{3}$" shift={(-2mm,5mm)}] {angle=d--c--i};
    \draw pic[fill=orange!50,angle radius=0.8cm,"$\alpha_{5}$" shift={(-6mm,-7.5mm)}] {angle=f--e--g};
    \draw pic[fill=orange!50,angle radius=0.8cm,"$\beta_{5}$" shift={(-15mm,-2mm)}] {angle=i--e--d};
    \draw pic[fill=orange!50,angle radius=0.6cm,"$\beta_{7}$" shift={(1mm,6mm)}] {angle=e--g--f};
    \draw[dashed, very thick] (0,0) -- (2,1) -- (2,3) -- (6,4) -- (4,2.5) -- (4.5,1);
    \draw[red, very thick] (0,0) -- (2,2) -- (0,4) -- (6,4) -- (4,0) -- cycle;
    \node at (3,1.5) {\Large $\widetilde{D}$};
    \node at (1.5,2) {\large \color{red} $\partial D$};
    \end{tikzpicture}
    \caption{An example setup for our problem with $\widetilde{D}$ highlighted in yellow and the angles formed by $\widetilde{D}$ and $D$ labelled at the vertices.}
    \label{fig:example_polygon}
\end{figure}

\begin{enumerate}
    \item[(i)] We say a vertex in $\widetilde{V}$ is an open vertex if it lies in $D$ and it has two incident edges in $\widetilde{E}$ which are open. Let $\mathcal{A}$ denote the collection of interior angles at all such open vertices.
    \item[(ii)] We say a vertex in $\widetilde{V}$ is a Neumann--Open--Neumann (NON) vertex if it lies on $\partial D$ and has two incident edges in $\widetilde{E}$, one of which is open and the other is Neumann. Let $\mathcal{B}$ denote the collection of pairs $(\gamma,\beta)$ at all such NON vertices where $\gamma$ denotes the interior angle of $\widetilde{D}$ at the vertex and $\beta$ denotes the exterior angle relative to $D$ at the vertex (see Figure \ref{fig:vertex_types}).
    \item[(iii)] We say a vertex in $\widetilde{V}$ is a Neumann--Open--Open--Neumann (NOON) vertex if it lies on $\partial D$ and
    it has two or four incident edges in $\widetilde{E}$, for which two are open and the rest are Neumann. Let $\mathcal{C}$ denote the collection of triples $(\gamma,\beta,\alpha)$ at all such NOON vertices where $\gamma$ denotes the angle between the open edges, which we call the middle angle at the vertex,
    and $\beta$ and $\alpha$ denote the two other exterior angles relative to $D$ at the vertex (see Figure \ref{fig:vertex_types}). The ordering of $\alpha$ and $\beta$ does not matter but we only have one triple for each NOON vertex. Note that NOON vertices either have that $\gamma$ is an interior angle of $\widetilde{D}$, or $\beta$ and $\alpha$ are both interior angles of $\widetilde{D}$.
     \item[(iv)] We say a vertex in $\widetilde{V}$ is a Neumann vertex if it lies on $\partial D$ and has two incident edges in $\widetilde{E}$ which are both Neumann.
\end{enumerate}

 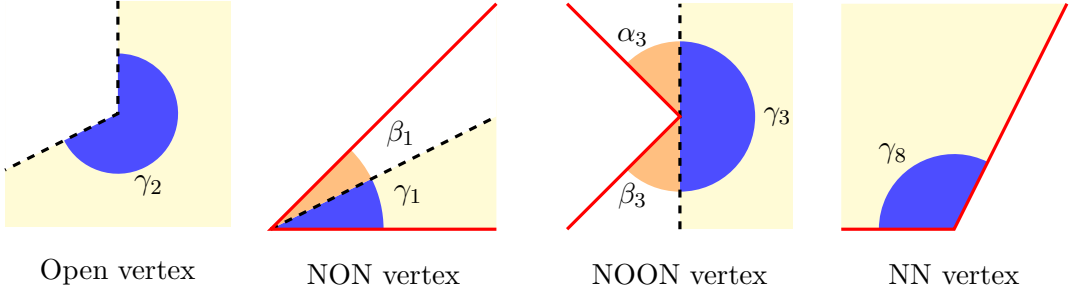
\begin{figure}
    \centering
    \begin{tikzpicture}[scale=3]
    \path[fill = yellow!20] (-0.5,1) -- (-0.5,0) -- (0.5,0) -- (0.5,1) -- cycle;
    \path [fill=white] (-0.5,0.25) -- (0,0.5) -- (0,1) -- (-0.5,1) -- cycle;
     \coordinate (a) at (0,0.5);
    \coordinate (b) at (0,1);
    \coordinate (c) at (-0.5,0.25);
    \draw pic[fill=blue!70,angle radius=0.8cm,"\large $\gamma_{2}$" shift={(0mm,-7mm)}] {angle=c--a--b};
    \draw[dashed, very thick] (0,1) -- (0,0.5) -- (-0.5,0.25);
    \node at (0,-0.2) {Open vertex};
    \end{tikzpicture} \quad
    \begin{tikzpicture}[scale=3]
    \path[fill=yellow!20] (1,0) -- (0,0) -- (1,1) -- cycle;
    \path [fill=white] (0,0) -- (1,0.5) -- (1,1) -- cycle;
     \coordinate (a) at (0,0);
    \coordinate (b) at (2,1);
    \coordinate (c) at (2,2);
    \coordinate (d) at (2,3);
    \coordinate (e) at (4,0);
    \draw pic[fill=blue!70,angle radius=1.5cm,"\large $\gamma_{1}$" shift={(9.5mm,2.5mm)}] {angle=e--a--b};
    \draw pic[fill=orange!50,angle radius=1.5cm,"$\beta_{1}$" shift={(10mm,7.5mm)}] {angle=b--a--c};
    \draw[dashed, very thick] (0,0) -- (1,0.5);
    \draw[red, very thick] (1,0) -- (0,0) -- (1,1);
    \node at (0.5,-0.2) {NON vertex};
    \end{tikzpicture} \qquad
    \begin{tikzpicture}[scale=3]
     \coordinate (a) at (0,0.5);
    \coordinate (b) at (0,0);
    \coordinate (c) at (0,1);
    \coordinate (d) at (-0.5,1);
    \coordinate (e) at (-0.5,0);
    \path[fill=yellow!20] (0,0) -- (0,1) -- (0.5,1) -- (0.5,0) -- cycle;
    \path [fill=white] (e) -- (a) -- (b) -- cycle;
    \path [fill=white] (d) -- (a) -- (c) -- cycle;
    \draw pic[fill=blue!70,angle radius=1cm,"$\gamma_{3}$" shift={(7mm,0mm)}] {angle=b--a--c};
    \draw pic[fill=orange!50,angle radius=1cm,"$\alpha_{3}$" shift={(-4mm,5mm)}] {angle=c--a--d};
    \draw pic[fill=orange!50,angle radius=1cm,"$\beta_{3}$" shift={(-4mm,-5mm)}] {angle=e--a--b};
    \draw[dashed, very thick] (0,0) -- (0,1);
    \draw[red, very thick] (d) -- (a) -- (e);
    \node at (0,-0.2) {NOON vertex};
    \end{tikzpicture} \quad
    \begin{tikzpicture}[scale=3]
     \coordinate (a) at (0,0);
    \coordinate (b) at (0.5,1);
    \coordinate (c) at (-0.5,0);
    \path[fill=yellow!20] (b) -- (a) -- (c) -- (-0.5,1) -- cycle;
    \draw pic[fill=blue!70,angle radius=1cm,"\large $\gamma_{8}$" shift={(-5mm,5mm)}] {angle=b--a--c};
    \draw[red, very thick] (c) -- (a) -- (b);
    \node at (0,-0.2) {NN vertex};
    \end{tikzpicture}
    \caption{An example of each of the four possible types of vertices that can arise in the problem. They are magnifications of the vertices in Figure \ref{fig:example_polygon} (see angle subscripts).
    See Figure \ref{fig:relected_iso_heat} for an example of the other type of NOON vertex.}
    \label{fig:vertex_types}
\end{figure}

Let $\partial_{O}\widetilde{D}$ denote the collection of open edges in $\widetilde{E}$. Our main result is the following.
\begin{thm}\label{thm:main}
There exists a constant $C_{D,\widetilde{D}} > 0$, depending only on $D$ and $\widetilde{D}$, such that
\begin{equation}\label{eq:main}
\begin{split}
    H_{D, \widetilde{D}}(t) & = \vert \widetilde{D} \vert
    - \frac{L(\partial_{O} \widetilde{D})}{\pi^{1/2}} t^{1/2} \\ & \qquad  + \left( \sum_{\gamma \in \mathcal A} a(\gamma)
    + \sum_{(\gamma,\beta)\in \mathcal{B}} b(\gamma,\beta) + \sum_{(\gamma,\beta,\alpha) \in \mathcal{C}} c(\gamma,\beta,\alpha)\right) t  + O\left(e^{-C_{D,\widetilde{D}}/t}\right),
\end{split}
\end{equation} where: $a : (0,2\pi) \to \mathbb{R}$ is defined as above in \eqref{eq:e1}; $b(\gamma,\beta)$ is given by \begin{equation}
    b(\gamma,\beta) := \int_{0}^{\infty} d\theta \; \frac{\cosh\left(\frac{\pi}{2}\theta\right)\cosh\left(\left(\gamma-\beta\right)\theta\right)-\cosh\left(\left(\frac{\pi}{2}-\gamma-\beta\right)\theta\right)}{2\sinh\left(\left(\gamma+\beta\right)\theta\right)\sinh\left(\frac{\pi}{2}\theta\right)};\end{equation} and $c(\gamma,\beta,\alpha)$ is given by \begin{equation}
        \begin{split}
            c(\gamma,\beta,\alpha) & := b(\gamma+\beta,\alpha) + b(\gamma+\alpha,\beta) \\
            & \quad + \int_{0}^{\infty} d\theta \; \frac{\cosh\left(\frac{\pi}{2}\theta\right)\left[\cosh\left(\left(\beta+\alpha\right)\theta\right)-\cosh\left(\left(\beta-\alpha\right)\theta\right)\right]}{\sinh\left(\left(\gamma+\beta+\alpha\right)\theta\right)\sinh\left(\frac{\pi}{2}\theta\right)}.
        \end{split}
    \end{equation}
\end{thm}

\begin{rem}
If $D$ were to have finitely many connected components, then the heat content of $\widetilde{D}$ would be the sum of its heat content on each of the connected components and thus Theorem \ref{thm:main} can be applied on each connected component. Theorem \ref{thm:main} can also be generalised to apply to vertices with an arbitrary number of incident edges (See Appendix A) and to the case where $D$ can have interior angles of angle $2\pi$ or where the heat content of $\widetilde{D}$ is considered in $\mathbb{R}^2 \setminus \partial D_{+}$ where $\partial D_{+} \subset \partial D$ subject to a generalised Neumann boundary condition (see Appendix B). We omit proofs of these in the main body of the paper for a clearer presentation.
\end{rem}

It was observed in \cite{vdBG16} that the function $a(\gamma)$, which corresponds to the contribution from the angles at an open vertex, is symmetric with respect to $\pi$. We observe that the function $b(\gamma, \beta)$, which corresponds to the contribution from the angles at a NON vertex, is symmetric in $\beta$ and $\gamma$, and that the function $c(\gamma,\beta,\alpha)$, which corresponds to the contribution from the angles at a NOON vertex, is symmetric in $\alpha$ and $\beta$ as mentioned above in (iii).
By these symmetry properties of $a(\gamma)$, $b(\gamma, \beta)$, and $c(\gamma, \beta, \alpha)$, we have the following consequence of Theorem \ref{thm:main}. If $\widetilde{D}_{1},\widetilde{D}_{2}$ are two polygonal subdomains of a bounded polygonal domain $D$ such that $\widetilde{D}_{2} = D \setminus \overline{\widetilde{D}_{1}}$, then
\begin{equation}\label{eqn:isoflow_condition}
    \left| \left(H_{D,\widetilde{D}_{1}}(t) - |\widetilde{D}_{1}|\right)-\left(H_{D,\widetilde{D}_{2}}(t) - |\widetilde{D}_{2}|\right)\right| = O(e^{-C/t})
\end{equation} for some constant $C>0$. Moreover, if $|\widetilde{D}_{1}|=|\widetilde{D}_{2}|$, then the heat contents of $\widetilde{D}_{1}$, $\widetilde{D}_{2}$ have the same long-time asymptotic behaviour and the same small-time asymptotic behaviour up to an exponentially small remainder.

Considering reflections of NON and NOON wedges with respect to a Neumann edge motivates the following relations between $a(\gamma)$, $b(\gamma,\beta)$, and $c(\gamma,\beta,\alpha)$.
\begin{prop}
For $\gamma,\beta,\alpha \in (0,\pi)$, we have the following.
\begin{itemize}
    \item[(i)] If $\gamma + \beta < \pi$, then $
    c(2\gamma,\beta,\beta) = 2b(\gamma,\beta)$; \vspace{0.5em}
    \item[(ii)]
    $b(\gamma,\pi-\gamma) = \frac{1}{2}a(2\gamma)$; \vspace{0.5em}
    \item[(iii)] If $\gamma + \beta + \alpha = \pi$ and $\alpha \leq \beta$, then $2 c(\gamma,\beta,\alpha) = 2a(\gamma) + 2k(2\alpha,\gamma,\gamma)$ and $2c(\gamma,\beta,\alpha) =  a(2\alpha) + a(2\beta) + 2k(\gamma,2\alpha,2\beta)$,
where the function $k(\alpha,\theta,\sigma)$ is given in \cite{vdBG16} as part of the open case when more than two open edges meet at a vertex, and is defined as
\begin{equation}
\begin{split}
    k(\alpha, \theta,\sigma) & := \frac{1}{2\pi}\big[-(\sigma+\theta+\alpha - \pi)\cot(\sigma+\theta+\alpha)-(\alpha-\pi) \cot(\alpha) \\
    & \qquad \quad + (\sigma+\alpha-\pi)\cot(\sigma+\alpha) + (\theta+\alpha-\pi)\cot(\theta+\alpha)\big]
\end{split}
\end{equation} for $\sigma+\theta+\alpha \neq \pi$, $\alpha \neq \pi$, $\sigma+\alpha \neq \pi$, $\theta+\alpha \neq \pi$. In any of the remaining cases, such as $\alpha = \pi$, $k(\alpha,\theta,\sigma)$ is defined by taking appropriate limits via l'H\^{o}pital's rule.
\end{itemize}
\end{prop}

\begin{rem}
These identities ultimately arise from how one can obtain Neumann heat kernels from the method of images. In particular, (ii) and (iii) arise from the Neumann heat kernel for the half-plane and so we have the relation with the angular contributions in the open case.
\end{rem}

\begin{proof}
We prove (i) by direct calculation, namely
\begin{equation}
        \begin{split}
            b(\gamma,\beta) & = \int_{0}^{\infty} d\theta \; \frac{\cosh\left(\frac{\pi}{2}\theta\right)\cosh\left(\left(\gamma-\beta\right)\theta\right)-\cosh\left(\left(\frac{\pi}{2}-\gamma-\beta\right)\theta\right)}{2\sinh\left(\left(\gamma+\beta\right)\theta\right)\sinh\left(\frac{\pi}{2}\theta\right)} \\
            & = \int_{0}^{\infty} d\theta \; \frac{\cosh\left(\frac{\pi}{2}\theta\right)\left(\cosh(2\gamma\theta)+\cosh(2\beta\theta)-1\right)-\cosh\left(\left(\frac{\pi}{2}-2\gamma-2\beta\right)\theta\right)}{2\sinh\left(2\left(\gamma+\beta\right)\theta\right)\sinh\left(\frac{\pi}{2}\theta\right)} \\
            & = b(2\gamma+\beta,\beta) + \int_{0}^{\infty} d\theta \; \frac{\cosh\left(\frac{\pi}{2}\theta\right)\left[\cosh\left(2\beta\theta\right)-1\right]}{2\sinh\left(2\left(\gamma+\beta\right)\theta\right)\sinh\left(\frac{\pi}{2}\theta\right)} \\
            & = \frac{1}{2} c(2\gamma,\beta,\beta).
        \end{split}
\end{equation}

For (ii) and (iii) we require an additional identity. For $|z| < |\pi|$, we have that by formula 3.511.9 in \cite{GR15}

\begin{equation}\label{eqn:cot_integral}
    \begin{split}
        \int_{0}^{\infty} d\theta \; \frac{\cosh(\frac{\pi}{2}\theta) \left(\cosh(z \theta)-1\right)}{\sinh(\pi \theta)\sinh(\frac{\pi}{2}\theta)} = \frac{2}{\pi} \int_{0}^{\infty} dx \; \frac{\sinh^{2}(\frac{z}{\pi}x)}{\sinh^{2}(x)} = \frac{1}{\pi} - \frac{z}{\pi}\cot(z).
    \end{split}
\end{equation}

We observe that (ii) clearly holds for $\gamma = \frac{\pi}{2}$. The case $\gamma \neq \frac{\pi}{2}$ is immediate from \eqref{eqn:cot_integral} by
\begin{equation}
    b(\gamma,\pi-\gamma) = \int_{0}^{\infty} d\theta \; \frac{\cosh(\tfrac{\pi}{2}\theta)\left(\cosh((\pi-2\gamma)\theta)-1\right)}{2\sinh(\pi\theta)\sinh(\tfrac{\pi}{2}\theta)}  = \frac{1}{2\pi} + \frac{1}{2}(1-\tfrac{2\gamma}{\pi})\cot(2\gamma) = \frac{1}{2}a(2\gamma).
\end{equation}

For (iii) it is sufficient to show that this identity holds when $2\gamma+2\alpha \neq \pi$, $2\alpha \neq \pi$, $
\gamma + 2\alpha \neq \pi$. By four uses of \eqref{eqn:cot_integral}, we see that

\begin{equation}
\begin{split}
    a(\gamma) + k(2\alpha,\gamma,\gamma) & = a(\gamma) + \frac{1}{2\pi}\big[-(\gamma+\alpha - \beta)\cot(\gamma+\alpha - \beta)-(\gamma+\beta-\alpha) \cot(\gamma+ \beta -\alpha) \\
    & \qquad \qquad \qquad + 2(\beta - \alpha)\cot(\beta-\alpha)\big] \\
    & = \int_{0}^{\infty} d\theta \; \frac{\cosh(\frac{\pi}{2}\theta) \left(\cosh((\gamma-\beta+\alpha) \theta)-1\right)}{2\sinh(\pi \theta)\sinh(\frac{\pi}{2}\theta)} \\
    & \quad + \int_{0}^{\infty} d\theta \; \frac{\cosh(\frac{\pi}{2}\theta)  \left(\cosh((\gamma+\beta-\alpha) \theta)-1\right)}{2\sinh(\pi \theta)\sinh(\frac{\pi}{2}\theta)} \\
    & \quad + \int_{0}^{\infty} d\theta \; \frac{\cosh(\frac{\pi}{2}\theta)  \left(\cosh(\pi-\gamma) \theta)-\cosh(\beta-\alpha) \theta)\right)}{\sinh(\pi \theta)\sinh(\frac{\pi}{2}\theta)} \\
    & = b(\gamma+\beta,\alpha) + b(\gamma+\alpha,\beta) \\
    & \quad + \int_{0}^{\infty} d\theta \; \frac{\cosh(\frac{\pi}{2}\theta)  \left(\cosh(\beta+\alpha) \theta)-\cosh(\beta-\alpha) \theta)\right)}{\sinh(\pi \theta)\sinh(\frac{\pi}{2}\theta)} \\
    & = c(\gamma,\beta,\alpha),
\end{split}
\end{equation} as desired. The other identity for $c(\gamma, \beta, \alpha)$ follows similarly.
\end{proof}

As a consequence of this proposition, for a rectangle $R$ with polygonal subdomain $\widetilde{D}$, we can choose one of its edges for reflection. We take the union of $R$ with its image under this reflection, call this $R_{1}$, and the union of $\widetilde{D}$ with its image under this reflection and call it $\widetilde{D}_{1}$. Then we observe that $\widetilde{D}_{1}$ is a polygonal subdomain of $R_{1}$ and
\begin{equation}
    \left| H_{R_{1},\widetilde{D}_{1}}(t) - 2 H_{R,\widetilde{D}}(t)\right| = O(e^{-C/t}). \end{equation}
Using this reflection principle, one can obtain polygonal subdomains of the unit square which have the same small-time heat content expansion up to an exponentially small remainder, see Figure \ref{fig:relected_iso_heat}.

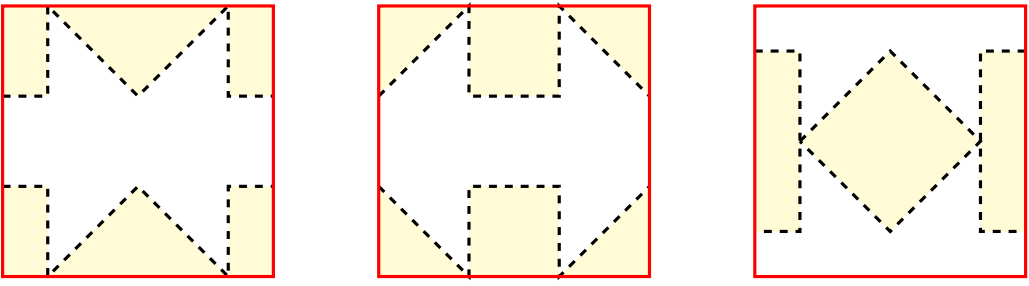
\begin{figure}
    \centering
    \begin{tikzpicture}
    \path[fill=white] (0,0) -- (0,3.6) -- (3.6,3.6) -- (3.6,0);
    \path[fill=white,xshift =5cm] (0,0) -- (0,3.6) -- (3.6,3.6) -- (3.6,0);
    \path[fill=white,xshift=10cm] (0,0) -- (0,3.6) -- (3.6,3.6) -- (3.6,0);

    \path [fill=yellow!20] (0,2.4) -- (0.6,2.4) -- (0.6,3.6) -- (1.8,2.4) -- (3,3.6) -- (3,2.4) -- (3.6,2.4) -- (3.6,3.6) -- (0,3.6) -- cycle;
    \path [fill=yellow!20] (0,1.2) -- (0.6,1.2) -- (0.6,0) -- (1.8,1.2) -- (3,0) -- (3,1.2) -- (3.6,1.2) -- (3.6,0) -- (0,0) -- cycle;
    \draw[black, very thick, dashed] (0,2.4) -- (0.6,2.4) -- (0.6,3.6) -- (1.8,2.4) -- (3,3.6) -- (3,2.4) -- (3.6,2.4);
    \draw[black, very thick, dashed] (0,1.2) -- (0.6,1.2) -- (0.6,0) -- (1.8,1.2) -- (3,0) -- (3,1.2) -- (3.6,1.2);
    \draw[red, very thick] (0,0) -- (0,3.6) -- (3.6,3.6) -- (3.6,0) -- cycle;

    \path [xshift = 5cm,fill=yellow!20] (0,2.4) -- (1.2,3.6) -- (1.2,2.4) -- (2.4,2.4) -- (2.4,3.6) -- (3.6,2.4) -- (3.6,3.6) -- (0,3.6) -- cycle;
    \path [xshift = 5cm,fill=yellow!20] (0,1.2) -- (1.2,0) -- (1.2,1.2) -- (2.4,1.2) -- (2.4,0) -- (3.6,1.2) -- (3.6,0) -- (0,0) -- cycle;
    \draw[xshift = 5cm,black,very thick,dashed] (0,2.4) -- (1.2,3.6) -- (1.2,2.4) -- (2.4,2.4) -- (2.4,3.6) -- (3.6,2.4);
    \draw[xshift = 5cm,black,very thick,dashed] (0,1.2) -- (1.2,0) -- (1.2,1.2) -- (2.4,1.2) -- (2.4,0) -- (3.6,1.2);
    \draw[xshift = 5cm, red, very thick] (0,0) -- (0,3.6) -- (3.6,3.6) -- (3.6,0) -- cycle;

    \path [xshift=10cm,fill=yellow!20] (0,3) -- (0.6,3) -- (0.6,0.6) -- (0,0.6) -- cycle;
    \path [xshift=10cm,fill=yellow!20] (3.6,3) -- (3,3) -- (3,0.6) -- (3.6,0.6) -- cycle;
    \path [xshift=10cm,fill=yellow!20] (0.6,1.8) -- (1.8,3) -- (3,1.8) -- (1.8,0.6) -- cycle;
    \draw[xshift=10cm,dashed,black, very thick] (0,3) -- (0.6,3) -- (0.6,0.6) -- (0,0.6);
    \draw[xshift=10cm,dashed,black, very thick] (3.6,3) -- (3,3) -- (3,0.6) -- (3.6,0.6);
    \draw[xshift=10cm,dashed,black, very thick] (0.6,1.8) -- (1.8,3) -- (3,1.8) -- (1.8,0.6) -- cycle;
    \draw[xshift=10cm,red, very thick] (0,0) -- (0,3.6) -- (3.6,3.6) -- (3.6,0) -- cycle;
    \end{tikzpicture}
    \caption{Three polygonal subdomains of the square with the same area and same small-time heat content expansion, up to an exponentially small remainder.}
    \label{fig:relected_iso_heat}
\end{figure}

\textbf{Outline of the paper: }
In Section 2, we detail the strategy of the proof of Theorem \ref{thm:main} using a partition of $\widetilde{D}$.
In Section 3 we compute a model heat content contribution for each part of the partition using explicit formulae for the heat kernel.
In Section 4 we prove that the difference between each of these model heat content contributions and the actual heat content contributions is exponentially small in small-time.
We note that the proofs in Section 3 rely on analytic methods whereas those in Section 4 rely on probabilistic ones via the relationship between the heat equation and Brownian motion. In the appendices, we discuss some generalisations of Theorem \ref{thm:main}.

\section{Proof of Theorem \ref{thm:main}}

The proof follows the strategy employed in the papers \cite{vdBS90}, \cite{vdBG16}, and \cite{vdBGG20}. The main idea is to partition $D$ in such a way that we can model the heat content contribution of each part of the partition by the heat content of the same part in a different ambient space to $\Omega$ for which a suitable, locally comparable heat kernel is known explicitly.

Let $\widetilde{V}$ be as defined in the last section and $V$ denote the vertices of $D$. Our first goal is to create an open sector, or the union of two open sectors at some NOON vertices, based at each vertex in $\widetilde{V}$ given by $\widetilde{S}_{v}(R):=B_{R}(v)\cap \widetilde{D}$ for some $R>0$, where $B_{R}(v)$ is the open disc in $\R^2$ centred at $v$ of radius $R$. We define the quantities \begin{equation}
    R_{1} := \frac{1}{2}\inf_{u \in \widetilde{V}\cup V} \inf_{\substack{v\in \widetilde{V}\cup V \\ v \neq u}} d(u,v), \qquad R_{2} := \frac{1}{2}\inf_{ u \in \widetilde{V}\setminus V} d(u,\partial D),
\end{equation} and set $R := \min\lbrace R_{1},R_{2}\rbrace$. The definition of $R_{1}$ ensures two things: $\widetilde{S}_{u}(R) \cap \widetilde{S}_{v}(R) = \emptyset$ for $u,v\in \widetilde{V}$ with $u\neq v$ and that $S_{u}(R) \cap S_{v}(R) = \emptyset$ also in this case, where $S_{v}(R) := B_{R}(v)\cap D$. The second point here is crucial for our comparisons. The definition of $R_{2}$ also allows us to do the required comparisons and calculations specifically for open vertices.

The next part of the partition we want to define is that of rectangles lying inside $\widetilde{D}$ for which one of its sides lies on $\partial \widetilde{D}$. Let $\widetilde{E}$ be as defined in the last section. For $\widetilde{e} \in \widetilde{E}$ with length $L(\widetilde{e})$, we want to have a rectangle of width $L(\widetilde{e}) - 2R$ and height $\delta > 0$ with a side of length $L(\widetilde{e})-2R$ lying on $\widetilde{e}$ and the rectangle lying inside $\widetilde{D}$. We denote this rectangle by $T_{\widetilde{e}}(R,\delta)$.
In order to apply our comparisons in the proceeding sections, we require that no two of these rectangles intersect so we must define a suitable choice of $\delta > 0$.
Denote the collection of interior angles of $\widetilde{D}$ by $\mathcal{A}_{1}$ and the collection exterior angles of $\widetilde{D}$ relative to $D$ at NON and NOON vertices by $\mathcal{A}_{2}$. Then define the quantities
$$\gamma_{1} := \left\{ \mu \in \mathcal{A}_{1} : \sin \frac{\mu}{2} = \min_{\kappa \in \mathcal{A}_{1}} \sin \frac{\kappa}{2},\right\}$$
and
$$\gamma_{2} := \left\{ \mu \in \mathcal{A}_{2} : \sin \frac{\mu}{2} = \min_{\kappa \in \mathcal{A}_{2}} \sin \frac{\kappa}{2},\right\}.$$
We then set $\delta_{1} := R \sin \frac{\gamma_{1}}{2}$ and $\delta_{2} := R\sin \frac{\gamma_{2}}{2}$. Then we determine $\delta$ by setting $\delta:= \min\lbrace \delta_{1},\delta_{2}\rbrace$. The definition of $\delta_{1}$ ensures that these rectangles do not overlap and the definition $\delta_{2}$ allows us to make our comparisons.

Now let $\widetilde{D}(R,\delta) := \lbrace x\in \widetilde{D} : d(x,\partial \widetilde{D}) > \delta \text{ and } d(x,\widetilde{V}) > R\rbrace$. Then we now must crucially observe that \begin{equation}
    \widetilde{D} \neq \left(\bigcup_{v\in \widetilde{V}}\widetilde{S}_{v}(R)\right) \cup \left(\bigcup_{\widetilde{e}\in \widetilde{E}}T_{\widetilde{e}}(R,\delta)\right)\cup \widetilde{D}(R,\delta)
\end{equation} even up to a set of measure zero so we have one final model space to consider. The remainder is the union of disjoint cusps which up to rigid planar motions are expressible as $\lbrace  x = (x_{1},x_{2}) \in \mathbb{R}^{2} : 0 < x_{1} < R, \vert x \vert > R, 0<x_{2} < \delta \rbrace $. The definition of $\delta$ ensures no two of these cusps overlap. Each sector $S_{v}(R)$ has two such cusps associated with it which we denote by $C_{v}^{(1)}(R,\delta)$ and $C_{v}^{(2)}(R,\delta)$. Then, up to a set of measure zero, we have that
    \begin{equation}
    \widetilde{D} = \left(\bigcup_{v\in \widetilde{V}}\widetilde{S}_{v}(R)\cup C_{v}^{(1)}(R,\delta)\cup C_{v}^{(2)}(R,\delta)\right) \cup \left(\bigcup_{\widetilde{e}\in \widetilde{E}}T_{\widetilde{e}}(R,\delta)\right)\cup \widetilde{D}(R,\delta).
\end{equation}

Now we have this partition, Theorem \ref{thm:main} follows immediately from the following theorem, which is subsequently proved in detail throughout the rest of this paper.

\begin{thm}\label{thm:model_justifications}

The following quantities are as defined above.
\begin{enumerate}[label=(\roman*)]
    \item  \begin{equation}
        \int_{\widetilde{D}(R,\delta)}dx \int_{\widetilde{D}}  dy\; \eta_{D}(t;x,y) = |\widetilde{D}(R,\delta)| + O\left(e^{-\delta^{2}/8t}\right).
    \end{equation}
    \item If $\widetilde{e} \in \widetilde{E}$ lies on $\partial D$, i.e. is a Neumann edge, then \begin{equation}
        \int_{T_{\widetilde{e}}(R,\delta)}dx\int_{\widetilde{D}} dy\; \eta_{D}(t;x,y) = |T_{\widetilde{e}}(R,\delta)| + O\left(e^{-\delta^{2}/8t}\right).
    \end{equation}
    \item If $\widetilde{e} \in \widetilde{E}$ lies in $D$, i.e. is an open edge, then \begin{equation}
        \int_{T_{\widetilde{e}}(R,\delta)} dx\int_{\widetilde{D}} dy\; \eta_{D}(t;x,y) = |T_{\widetilde{e}}(R,\delta)| - \frac{(L(\widetilde{e})-2R)}{\pi^{1/2}}t^{1/2} + O\left(e^{-\delta^{2}/8t}\right).
    \end{equation}
    \item If $C(R,\delta)$ is a cusp lying adjacent to an edge $\widetilde{e}\in\widetilde{E}$ lying on $\partial D$, i.e. is a Neumann cusp, then we have that \begin{equation}
        \int_{C(R,\delta)} dx\int_{\widetilde{D}} dy\; \eta_{D}(t;x,y) = |C(\delta,R)| + O\left(e^{-\delta^{2}/8t}\right).
    \end{equation}
    \item If $C(R,\delta)$ is a cusp lying adjacent to an edge $\widetilde{e}\in\widetilde{E}$ lying in $D$, i.e. is an open cusp, then we have that \begin{equation}
    \begin{split}
        \int_{C(R,\delta)} dx\int_{\widetilde{D}} dy\; \eta_{D}(t;x,y) & = |C(\delta,R)| - \frac{R}{\pi^{1/2}}t^{1/2}\int_{1}^{\infty}\frac{dv}{v^{2}}\int_{0}^{1}dy\frac{y}{(1-y^{2})^{\frac{1}{2}}} e^{-\frac{R^{2}y^{2}v^{2}}{4t}} \\
        & \quad +  O\left(e^{-\delta^{2}/8t}\right).
    \end{split}
    \end{equation}
    \item For $v\in \widetilde{V}$ an NN vertex \begin{equation}
        \int_{\widetilde{S}_{v}(R)} dx\int_{\widetilde{D}} dy\; \eta_{D}(t;x,y) = |\widetilde{S}_{v}(R)| +  O\left(e^{-\delta^{2}/8t}\right).
    \end{equation}
    \item For $v\in \widetilde{V}$ a NON vertex with interior angle $\gamma$ and exterior angle $\beta$, \begin{equation}
    \begin{split}
        \int_{\widetilde{S}_{v}(R)} dx\int_{\widetilde{D}} dy\; \eta_{D}(t;x,y) & = |\widetilde{S}_{v}(R)| - \frac{R}{\pi^{1/2}}t^{1/2} + b(\gamma,\beta)t  \\
        & \quad + \frac{R}{\pi^{1/2}}t^{1/2}\int_{1}^{\infty}\frac{dv}{v^{2}}\int_{0}^{1}dy\frac{y}{(1-y^{2})^{\frac{1}{2}}} e^{-R^{2}y^{2}v^{2}/4t} +
        O\left(e^{-C_{2}/t}\right),
    \end{split}
    \end{equation} where $C_{2} > 0$ is a constant depending on $R$, $\gamma$ and $\beta$.
    \item For $v\in \widetilde{V}$ a NOON vertex with middle angle $\gamma$ and exterior angles $\beta,\alpha$,
    \begin{equation}
    \begin{split}
        \int_{\widetilde{S}_{v}(R)} dx\int_{\widetilde{D}} dy\; \eta_{D}(t;x,y) & = |\widetilde{S}_{v}(R)| - \frac{2R}{\pi^{1/2}}t^{1/2} + c(\gamma,\beta,\alpha)t  \\
        & \quad + \frac{2R}{\pi^{1/2}}t^{1/2}\int_{1}^{\infty}\frac{dv}{v^{2}}\int_{0}^{1}dy\frac{y}{(1-y^{2})^{\frac{1}{2}}} e^{-R^{2}y^{2}v^{2}/4t} +
        O\left(e^{-C_{3}/t}\right),
    \end{split}
    \end{equation} where $C_{3} > 0$ is a constant depending on $R$, $\gamma$, $\beta$ and $\alpha$.
\end{enumerate}
\end{thm}

The key point to note is that: sectors $\widetilde{S}_{v}(R)$ at
NN vertices have two associated Neumann cusps which have trivial contribution,
sectors $\widetilde{S}_{v}(R)$ at NON vertices have one associated Neumann cusp and one associated open cusp which cancels out the term
\begin{equation}
    \frac{R}{\pi^{1/2}}t^{1/2}\int_{1}^{\infty}\frac{dv}{v^{2}}\int_{0}^{1}dy\frac{y}{(1-y^{2})^{\frac{1}{2}}} e^{-R^{2}y^{2}v^{2}/4t},
\end{equation} and
sectors $\widetilde{S}_{v}(R)$ at NOON vertices have two associated open cusps which cancel out the term \begin{equation}
    \frac{2R}{\pi^{1/2}}t^{1/2}\int_{1}^{\infty}\frac{dv}{v^{2}}\int_{0}^{1}dy\frac{y}{(1-y^{2})^{\frac{1}{2}}} e^{-R^{2}y^{2}v^{2}/4t}.
\end{equation}
Analogous results also hold for the case of an open vertex with two neighbouring cusps. Indeed, we recall the following results from \cite{vdBG16}.
\begin{lemma}[Open vertex with two open cusps {\cite[Lem. 9 \& \S 4.2]{vdBG16}}]\label{open_wedge}
We have that \begin{equation}
\begin{split}
    & \int_{0}^{R} dr \; r \int_{0}^{\gamma} d\phi \int_{0}^{\infty} dr_{0} \; r_{0} \int_{0}^{\gamma} d\phi_{0}\; p_{\mathbb{R}^{2}}(t;r,\phi,r_{0},\phi_{0})
    + 2\int_{C(R,\delta)} dx\int_{\widetilde{D}} dy\; p_{\mathbb{R}^{2}}(t;x,y)\\
    & \qquad \qquad \qquad =  |\widetilde{S}_{v}(R)| +|C(R,\delta)| - \frac{2R}{\pi^{1/2}}t^{1/2} + a(\gamma) t
    + O\left(te^{-R^{2}C_{\gamma,\delta}/t}\right),
\end{split}
\end{equation} where $C_{\gamma,\delta} > 0$ is a constant depending only on $\gamma, \delta$.
\end{lemma}
Hence when one sums up all the heat content contributions from each part of the partition we get the desired form.

\section{Model computations of heat content}

In this section we prove some explicit heat content calculations that will act as our model heat content contributions. This will prove half of Theorem \ref{thm:model_justifications} and give the explicit coefficients with respect to the geometry that we are interested in. Full justification of why these approximations are valid, and thus completing the proof of Theorem \ref{thm:model_justifications}, is given in the next section but for this section we will only give heuristic explanations so as to aid with the intuition for the problem.

From the construction of our partition in the previous section we see that $\widetilde{D}(R,\delta)$ is compactly contained in $\widetilde{D}$ and thus in $D$ and so in small-time heat should flow similar to as it would if we simply replaced $D$ by $\mathbb{R}^{2}$. Thus, the model content contribution from $\widetilde{D}(R,\delta)$ is given by \begin{equation}
    \int_{\widetilde{D}(R,\delta)}dx\int_{\widetilde{D}} dy \; p_{\R^{2}}(t;x,y)
\end{equation} where $p_{\R^{2}}(t;x,y)$ is the heat kernel for $\R^{2}$. For any $x\in \widetilde{D}(R,\delta)$, we have that $d(x,\partial \widetilde{D}) \geq \delta$ hence\begin{equation}
    \begin{split}
        1\geq \int_{\widetilde{D}} dy\; p_{\mathbb{R}^{2}}(t;x,y) & = \int_{\mathbb{R}^{2}} dy\; p_{\mathbb{R}^{2}}(t;x,y) - \int_{\mathbb{R}^{2}\backslash\widetilde{D}} dy \; p_{\mathbb{R}^{2}}(t;x,y) \\
        & = 1 - (4\pi t)^{-1} \int_{\mathbb{R}^{2}\backslash\widetilde{D}} dy\; e^{-\vert x-y\vert^{2}/4t} \\
        & \geq  1-(4\pi t)^{-1} e^{-\frac{\delta^{2}}{8t}} \int_{\mathbb{R}^{2}}dy \; e^{-\vert x-y\vert^{2}/8t} \\
        & = 1 - 2 e^{-\delta^{2}/8t}.
    \end{split}
\end{equation}
This is the analogue of the `principle of not feeling the boundary' that was introduced in \cite{K66}, see also \cite[Prop. 9(i)]{vdB13}. The idea is that in small-time on the interior we should not detect any heat loss.
Hence, as in \cite[Lem. 4]{vdBG16}, it follows that \begin{equation}
    \left\vert  \int_{\widetilde{D}(R,\delta)}dx\int_{\widetilde{D}} dy \; p_{\R^{2}}(t;x,y) - \vert\widetilde{D}(R,\delta)\vert \right\vert \leq 2 e^{-\delta^{2}/8t}.
\end{equation}

For an edge $\widetilde{e}\in \widetilde{E}$ lying in $D$, in small-time the rectangle $T_{\widetilde{e}}(R,\delta)$ should almost look like it is living in $\mathbb{R}^{2}$ but instead in the case that the initial datum is the indicator function of the half-plane containing $T_{\widetilde{e}}(R,\delta)$ whose boundary contains $\widetilde{e}$. Up to a rigid planar motion, this reads:

\begin{lemma}[{\cite[\S 4.1]{vdBG16}}]\label{lem:model_open_hp}
Let $T = (0,L(\widetilde{e})-2R) \times (0,\delta)$. Then we have that \begin{equation}
    \int_{T}dx\int_{\mathbb{H}}dy\; p_{\mathbb{R}^{2}}(t;x,y) = |T| - \frac{L(\widetilde{e})-2R}{\pi^{1/2}}t^{1/2}+O\left(t^{1/2}e^{-\delta^{2}/8t}\right),
\end{equation} where $\mathbb{H}$ denotes the half-plane $\mathbb{R} \times \mathbb{R}_{> 0}$.
\end{lemma}

The same can be done for a cusp $C(R,\delta)$ lying adjacent to an open edge and here one has:

\begin{lemma}[{\cite[Lem. 2.3]{vdBGG20}}] \label{lem:open_cusp}
\begin{equation}
\begin{split}
    \int_{C(R,\delta)} dx \int_{\mathbb{H}} dy \; p_{\mathbb{R}^{2}}(t;x,y) & = |C(R,\delta)| - \frac{R}{\pi^{1/2}}t^{1/2}\int_{1}^{\infty}\frac{dv}{v^{2}}\int_{0}^{1}dy\frac{y}{(1-y^{2})^{1/2}} e^{-R^{2}y^{2}v^{2}/4t} \\ &\quad  + O\left(t^{1/2}e^{-\delta^{2}/4t}\right).
\end{split}
\end{equation}
\end{lemma}

Now if $\widetilde{e}\in \widetilde{E}$ lies on $\partial D$, then in small-time both the rectangle $T_{\widetilde{e}}(R,\delta)$ and any cusp $C(R,\delta)$ should look like they are lying in a half-plane with Neumann boundary condition imposed and initial datum of the indicator function of the half-plane.

Explicitly, one can obtain the heat kernel of the half-plane $\mathbb{H}$ with Neumann boundary condition imposed to be \begin{equation}
    \eta_{\mathbb{H}}(t;x,y) := p_{\mathbb{R}^{2}}(t;x,y) + p_{\mathbb{R}^{2}}(t;x,y^{*})
\end{equation} where $y^{*}=(y_{1},-y_{2})$. It is immediate from standard Gaussian integrals that \begin{equation}
    \int_{\mathbb{H}}dy \;\eta_{\mathbb{H}}(t;x,y) = 1
\end{equation} and hence for $A = T_{\widetilde{e}}(R,\delta)$ or $A = C(R,\delta)$, we see that \begin{equation}
    \int_{A}dx\int_{\mathbb{H}}dy \;\eta_{\mathbb{H}}(t;x,y) = |A|.
\end{equation} So, as expected from physical intuition, we see that the phenomenon of not feeling the boundary does also occur near edges with solely Neumann boundary condition imposed.

The remaining model computations are those of the sectors $\widetilde{S}_{v}(R)$, however these are much more technically involved. For a vertex $v\in \widetilde{V}$ lying on $\partial D$ we model the heat content of the sector $\widetilde{S}_{v}(R)$, by the heat content of the same sector but instead lying in an infinite wedge with Neumann boundary conditions imposed obtained by stretching $S_{v}(R)$ out to infinity radially outwards from $v$ and initial datum of $\widetilde{S}_{v}(R)$ stretched out to infinity in the same way. The idea is that the heat content of $\widetilde{S}_{v}(R)$ in each situation should be similar.

For $0 < \gamma < 2\pi$, the Green's function for the heat equation with Neumann boundary conditions on the infinite wedge $W_{\gamma}:= \lbrace (r,\phi) : r > 0, 0 < \phi < \gamma\rbrace $, that is the solution to \begin{equation}
    \begin{cases}
    sG_{W_{\gamma}} - \partial_{rr}G_{W_{\gamma}} - r^{-1}\partial_{r}G_{W_{\gamma}} - r^{-2}\partial_{\phi\phi}G_{W_{\gamma}} = r^{-1}\delta(r-r_{0})\delta(\phi-\phi_{0}), \\
    \partial_{\phi}G_{W_{\gamma}} = 0, \; \phi = 0,\gamma
    \end{cases}
\end{equation} can be computed explicitly to be \begin{equation}
    G_{W_{\gamma}}(s;r,\phi,r_{0},\phi_{0}) = \frac{1}{\pi^{2}} \int_{0}^{\infty} d\theta \; K_{i\theta}(r\sqrt{s}) K_{i\theta}(r_{0}\sqrt{s}) \Phi_{\gamma}(\theta,\phi,\phi_{0})
\end{equation} where \begin{equation}
\begin{split}
    \Phi_{\gamma}(\theta,\phi,\phi_{0}) & = \cosh\left(\left(\pi - |\phi_{0}-\phi|\right)\theta\right) +\frac{\sinh\left(\pi\theta\right)}{\sinh\left(\gamma\theta\right)} \cosh\left(\left(\phi+\phi_{0}-\gamma \right)\theta\right) \\
    & \qquad +\frac{\sinh\left(\left(\pi-\gamma\right)\theta\right)}{\sinh\left(\gamma\theta\right)} \cosh\left(\left(\phi-\phi_{0} \right)\theta\right),
\end{split}
\end{equation}
see, for example, \cite[Appendix A]{NRS19} and references therein.
Here the $K_{\nu}$ are modified Bessel functions of the second kind, that is the unique solution to the equation \begin{equation}
    z^{2}K_{\nu}''(z) + zK_{\nu}'(z)-(z^{2}+\nu^{2})K_{\nu}(z) = 0.
\end{equation}
We note that the above approach expresses the Green’s function of an infinite wedge as a Kontorovich Lebedev transform. This approach was used by D. B. Ray to compute the angular contribution to the small-time asymptotic expansion of the Dirichlet heat trace for a polygon (see the footnote on page 44 of \cite{MS67}) as well as in \cite{vdBGG20,vdBS88, vdBS90}.
The unique Neumann heat kernel $\eta_{W_{\gamma}}$ on $W_{\gamma}$ is given by the inverse Laplace transform of $G_{W_{\gamma}}$ i.e. \begin{equation}
    \eta_{W_{\gamma}}(t;r,\phi,r_{0},\phi_{0}) = \mathcal{L}^{-1}\left\lbrace G_{W_{\gamma}}(s;r,\phi,r_{0},\phi_{0}) \right\rbrace(t).
\end{equation}

By noting that the only solution to the heat equation on $W_{\gamma}$ with Neumann boundary condition and initial datum $\mathds{1}_{W_{\gamma}}$ is $u = \mathds{1}_{W_{\gamma}}$,
it is immediate that the heat content of $\widetilde{S}_{v}(R)$ when $v$ is an NN vertex is $|\widetilde{S}_{v}(R)|$. However, the latter route is not as illuminating as to our method when the computations become non-trivial. So we recall the relevant computations below to motivate the following lemmas.

\begin{lemma}\label{lem:NN}
For all $t>0$, \begin{equation}
    \int_{0}^{R} r \; dr \int_{0}^{\gamma} d\phi \int_{0}^{\infty} r_{0}\; dr_{0} \int_{0}^{\gamma}d\phi_{0} \; \eta_{W_{\gamma}}(t,r,\phi,r_{0},\phi_{0}) = \frac{1}{2}\gamma R^{2}.
\end{equation}
\end{lemma}

\begin{proof}
Observe that we are computing the quantity \begin{equation}
\begin{split}
    & \int_{0}^{R} r \; dr \int_{0}^{\gamma} d\phi \int_{0}^{\infty} r_{0}\; dr_{0} \int_{0}^{\gamma}d\phi_{0} \; \mathcal{L}^{-1}\left\lbrace G_{W_{\gamma}}(s,r,\phi,r_{0},\phi_{0}) \right\rbrace(t) \\
    & \quad = \int_{0}^{R} r \; dr \int_{0}^{\gamma} d\phi \int_{0}^{\infty} r_{0}\; dr_{0} \int_{0}^{\gamma}d\phi_{0} \;  \mathcal{L}^{-1}\bigg\lbrace \frac{1}{\pi^{2}} \int_{0}^{\infty} d\theta \; K_{i\theta}(r\sqrt{s}) K_{i\theta}(r_{0}\sqrt{s}) \Phi_{\gamma}(\theta,\phi,\phi_{0}) \bigg\rbrace(t)
\end{split}
\end{equation} and that by Fubini's theorem we can rearrange the integrals. One can easily compute that \begin{equation}
    \int_{0}^{\gamma} d\phi \int_{0}^{\gamma} d\phi_{0} \; \Phi_{\gamma}(\theta,\phi,\phi_{0}) = \frac{2\gamma}{\theta}\sinh\left(\pi\theta\right)
\end{equation} and hence we now want to compute the quantity \begin{equation}
    \frac{2\gamma}{\pi^{2}}\mathcal{L}^{-1}\left\lbrace \int_{0}^{R} r\; dr \int_{0}^{\infty} r_{0} \; dr_{0}\int_{0}^{\infty}\frac{d\theta}{\theta} \; K_{i\theta}(r\sqrt{s}) K_{i\theta}(r_{0}\sqrt{s})\sinh(\pi\theta)\right\rbrace(t).
\end{equation} This quantity has been computed before in \cite[\S 2]{vdBS90} and is  $\frac{1}{2}\gamma R^{2}$ as desired. For the sake of completeness we give the calculation here. Combining formulae 6.561.16 and 8.332.3 in \cite{GR15}, we see that \begin{equation}\label{bessel_trick}
    \int_{0}^{\infty} dr \; rK_{i\theta}(r\sqrt{s}) = \frac{\pi\theta}{2s\sinh\left(\tfrac{\pi}{2}\theta\right)}
\end{equation} and formula 6.794.2 in \cite{GR15} reads \begin{equation}
    \int_{0}^{\infty} d\theta \; K_{i\theta}(r\sqrt{s}) \cosh\left(\tfrac{\pi}{2}\theta\right) = \frac{\pi}{2}.
\end{equation} Hence, applying these successively we see that \begin{equation}\label{area_term}
    \begin{split}
        & \frac{2\gamma}{\pi^{2}}\mathcal{L}^{-1}\left\lbrace \int_{0}^{R} r\; dr \int_{0}^{\infty} r_{0} \; dr_{0}\int_{0}^{\infty}\frac{d\theta}{\theta} \; K_{i\theta}(r\sqrt{s}) K_{i\theta}(r_{0}\sqrt{s})\sinh(\pi\theta)\right\rbrace(t) \\
        & \qquad = \frac{2\gamma}{\pi}\mathcal{L}^{-1}\left\lbrace \frac{1}{s} \int_{0}^{R} r\; dr \int_{0}^{\infty}d\theta \; K_{i\theta}(r\sqrt{s}) \cosh(\tfrac{\pi}{2}\theta)\right\rbrace(t) \\
         & \qquad = \gamma \mathcal{L}^{-1}\left\lbrace \frac{1}{s} \int_{0}^{R} r\; dr \right\rbrace(t) = \frac{1}{2}\gamma R^{2}.
    \end{split}
\end{equation}
\end{proof}

The case of NON and NOON vertices is much more technical and will require the following technical tools that we shall now prove.

\begin{lemma}\label{searchlight_coverings}
Let $\mathcal{T}:=\lbrace (\sigma,\rho,\lambda): 0<\lambda<\rho < \sigma <2\pi \rbrace$. Let $\xi:\mathcal{T}\to \mathbb{R}$ be a function in only $\rho$ and $\lambda$, denoted $\xi(\rho,\lambda)$, such that $0\leq \xi(\rho,\lambda)<2\sigma$. \begin{enumerate}[label=(\roman*)]
\item We have that the collection $\lbrace A_{N}, B_{N}\rbrace_{N\in\mathbb{Z}_{\geq 0}}$, where \begin{equation}
    A_{N}:= \lbrace (\sigma,\rho,\lambda)\in\mathcal{T} : |\tfrac{\pi}{2}-(2N+1)\sigma + \xi(\rho,\lambda)|<\sigma\rbrace
\end{equation} and \begin{equation}
    B_{N}:= \lbrace (\sigma,\rho,\lambda)\in\mathcal{T} : \tfrac{\pi}{2}+\xi(\rho,\lambda) = 2(N+1)\sigma \rbrace,
\end{equation} forms a covering of $\mathcal{T}$. Moreover, for $N\geq 1$ fixed, any $(\sigma,\rho,\lambda) \in A_{N}$ or $(\sigma,\rho,\lambda) \in B_{N}$, and any $1\leq n \leq N$, we have that $0 < \tfrac{\pi}{2} - 2n\sigma +\xi(\rho,\lambda) < \tfrac{\pi}{2}$.

\item If $\xi(\rho,\lambda)>0$ then we have that the collection $\lbrace C_{N}, D_{N}\rbrace_{N\in\mathbb{Z}_{\geq 0}}$, where \begin{equation}
    C_{N}:= \lbrace (\sigma,\rho,\lambda)\in\mathcal{T} : |\tfrac{\pi}{2}-(2N-1)\sigma - \xi(\rho,\lambda)|<\sigma\rbrace
\end{equation} and \begin{equation}
    D_{N}:= \lbrace (\sigma,\rho,\lambda)\in\mathcal{T} : \tfrac{\pi}{2}-\xi(\rho,\lambda) = 2N\sigma \rbrace,
\end{equation} forms a covering of $\mathcal{T}$. Moreover, for $N\geq 1$ fixed, any $(\sigma,\rho,\lambda) \in C_{N}$ or $(\sigma,\rho,\lambda) \in D_{N}$, and any $1\leq n \leq N$, we have that $0 < \tfrac{\pi}{2} - 2(n-1)\sigma -\xi(\rho,\lambda) < \tfrac{\pi}{2}$.
\end{enumerate}
\end{lemma}

\begin{proof}
For (i) we have that $(\sigma,\rho,\lambda)\in A_{N}$ if and only if $-\sigma < \tfrac{\pi}{2}-(2N+1)\sigma + \xi(\lambda,\sigma) < \sigma$ which occurs if and only if $2N\sigma < \tfrac{\pi}{2} + \xi(\rho,\lambda) < 2(N+1)\sigma$.
Now pick an arbitrary point $(\sigma,\rho,\lambda)\in\mathcal{T}$. Suppose that $\tfrac{\pi}{2} + \xi(\rho,\lambda) \neq 2(M+1)\sigma$ for any $M\in\mathbb{Z}_{\geq 0}$. Then clearly there exists a $K\in\mathbb{Z}_{\geq 0}$ such that $2K\sigma < \tfrac{\pi}{2}+\xi(\rho,\lambda)<2(K+1)\sigma$ and hence $(\sigma,\rho,\lambda) \in A_{K}$. If we have that $\tfrac{\pi}{2} + \xi(\rho,\lambda) = 2(K'+1)\sigma$ for some $K'\in \mathbb{Z}_{\geq 0}$, then $(\sigma,\rho,\lambda) \in B_{K'}$. Since our choice of point $(\sigma,\rho,\lambda)$ was arbitrary, we indeed have that $\lbrace A_{N}, B_{N}\rbrace_{N\in\mathbb{Z}_{\geq 0}}$ is a covering of $\mathcal{T}$. Now fix $N\geq 1$. Let us choose an arbitrary $(\sigma,\rho,\lambda)\in A_{N}$. For $1\leq n \leq N$, we have that \begin{equation}
    \tfrac{\pi}{2} - 2n \sigma +\xi(\rho,\lambda) \leq \tfrac{\pi}{2} - 2 \sigma +\xi(\rho,\lambda) < \tfrac{\pi}{2}
\end{equation} and that \begin{equation}
    \tfrac{\pi}{2} - 2n \sigma +\xi(\rho,\lambda) \geq \tfrac{\pi}{2} - 2N \sigma +\xi(\rho,\lambda) > 2N\sigma -2N\sigma = 0.
\end{equation} Now let us choose $(\sigma,\rho,\lambda)\in B_{N}$. For $1\leq n \leq N$, we have that \begin{equation}
    \tfrac{\pi}{2} - 2n \sigma +\xi(\rho,\lambda) \leq \tfrac{\pi}{2} - 2 \sigma +\xi(\rho,\lambda) < \tfrac{\pi}{2}
\end{equation} and that \begin{equation}
    \tfrac{\pi}{2} - 2n \sigma +\xi(\rho,\lambda) \geq \tfrac{\pi}{2} - 2N \sigma +\xi(\rho,\lambda) = 2(N+1)\sigma -2N\sigma = 2\sigma > 0.
\end{equation} So we have proven (i).

We prove (ii) in a similar fashion. We have that $(\sigma,\rho,\lambda) \in C_{N}$ if and only if $2(N-1)\sigma < \tfrac{\pi}{2} - \xi(\rho,\lambda) < 2N\sigma$. Now pick an arbitrary point $(\sigma,\rho,\lambda)\in\mathcal{T}$. Suppose that $\tfrac{\pi}{2} - \xi(\rho,\lambda) \neq 2M\sigma$ for any $M\in\mathbb{Z}_{\geq 0}$. Then clearly there exists a $K\in\mathbb{Z}_{\geq 0}$ such that $2(K-1)\sigma < \tfrac{\pi}{2}-\xi(\rho,\lambda)<2K\sigma$ and hence $(\sigma,\rho,\lambda) \in C_{K}$. If we have that $\tfrac{\pi}{2} -\xi(\rho,\lambda) = 2K'\sigma$ for some $K'\in \mathbb{Z}_{\geq 0}$, then $(\sigma,\rho,\lambda) \in D_{K'}$. Since our choice of point $(\sigma,\rho,\lambda)$ was arbitrary, we indeed have that $\lbrace C_{N}, D_{N}\rbrace_{N\in\mathbb{Z}_{\geq 0}}$ is a covering of $\mathcal{T}$. Now fix $N\geq 1$. Let us choose an arbitrary $(\sigma,\rho,\lambda)\in C_{N}$. For $1\leq n \leq N$, we have that \begin{equation}
    \tfrac{\pi}{2} - 2(n-1) \sigma -\xi(\rho,\lambda) \leq \tfrac{\pi}{2} -\xi(\rho,\lambda) < \tfrac{\pi}{2}
\end{equation} and that \begin{equation}
    \tfrac{\pi}{2} - 2(n-1) \sigma -\xi(\rho,\lambda) \geq \tfrac{\pi}{2} - 2(N-1) \sigma -\xi(\rho,\lambda) > 2(N-1)\sigma -2(N-1)\sigma = 0.
\end{equation} Now let us choose an $(\sigma,\rho,\lambda)\in D_{N}$. For $1\leq n \leq N$, we have that \begin{equation}
    \tfrac{\pi}{2} - 2(n-1) \sigma -\xi(\rho,\lambda) \leq \tfrac{\pi}{2} - \xi(\rho,\lambda) < \tfrac{\pi}{2}
\end{equation} and that \begin{equation}
    \tfrac{\pi}{2} - 2(n-1) \sigma -\xi(\rho,\lambda) \geq \tfrac{\pi}{2} - 2(N-1) \sigma -\xi(\rho,\lambda) = 2N\sigma -2(N-1)\sigma = 2\sigma > 0.
\end{equation} So we have proven (ii).
\end{proof}

Before moving forward, let us state some known results for small-time asymptotics for inverse Laplace transforms.

\begin{lemma}\label{inverse_Laplace_transforms}
\begin{enumerate}[label=(\roman*)]
\item\label{cosh_sinh_convergent_integral} For $a,b\in \R$ such that $|a|<|b|$, we have that \begin{equation}
    \mathcal{L}^{-1}\left\lbrace \frac{1}{s}\int_{R}^{\infty} rdr\int_{0}^{\infty}d\theta K_{i\theta}(r\sqrt{s})\frac{\cosh(a\theta)-1}{\theta\sinh(b\theta)}\right\rbrace (t) = O\left(te^{-R^{2}/4t}\right).
\end{equation}
    \item\label{sinh_remainder} For $-\tfrac{\pi}{2}<a<\tfrac{\pi}{2}$, we have that \begin{equation}
    \mathcal{L}^{-1}\left\lbrace \frac{1}{s} \int_{R}^{\infty} rdr\int_{0}^{\infty}d\theta K_{i\theta}(r\sqrt{s})\frac{\sinh(a\theta)}{\theta}\right\rbrace (t) = O\left(te^{-R^{2}\cos^{2}(a)/4t}\right).
\end{equation}
\item\label{tanh_remainder} For $a>0$, we have that \begin{equation}
    \mathcal{L}^{-1}\left\lbrace \frac{1}{s} \int_{R}^{\infty} rdr\int_{0}^{\infty}d\theta K_{i\theta}(r\sqrt{s})\frac{\tanh(a\theta)}{\theta}\right\rbrace (t) = O\left(te^{-R^{2}/4t}\right).
\end{equation}
\end{enumerate}
\end{lemma}

\begin{proof}
Proofs of (ii) and (iii) can be found in \cite[\S 2]{vdBS90}. For (i), note that if $|a|<|b|$ then the integral \begin{equation}
    \int_{0}^{\infty} d\theta \left| \frac{\cosh(a\theta)-1}{\theta \sinh(b\theta)} \right| < \infty,
\end{equation} and we immediately obtain the result following similar methods to those in \cite[\S 4]{vdBGG20}.
\end{proof}

Using the previous two lemmata we can now prove the following crucial lemma, which is done in the spirit of the proof that the remainder is exponentially small in \cite{vdBS90}.

\begin{lemma}\label{angle_remainder}
Let $(\sigma,\rho,\lambda)\in \mathcal{T}$ and $\xi$ as in Lemma \ref{searchlight_coverings}. \begin{enumerate}[label=(\roman*)]
    \item \label{posxi} Then \begin{equation}
    \mathcal{L}^{-1}\left\lbrace \frac{1}{s} \int_{R}^{\infty} rdr\int_{0}^{\infty}d\theta K_{i\theta}(r\sqrt{s})\frac{\cosh((\frac{\pi}{2}-\sigma+\xi(\rho,\lambda))\theta)-1}{\theta\sinh(\sigma\theta)}\right\rbrace (t) = O\left(te^{-R^{2}C_{1}^{\sigma,\rho,\lambda}/4t}\right),
    \end{equation} where $C_{1}^{\sigma,\rho,\lambda}>0$ is a constant depending only on $\sigma$, $\rho$ and $\lambda$.
    \item\label{nnxi} If $\xi(\rho,\lambda) > 0$, then \begin{equation}
    \mathcal{L}^{-1}\left\lbrace \frac{1}{s} \int_{R}^{\infty} rdr\int_{0}^{\infty}d\theta K_{i\theta}(r\sqrt{s})\frac{\cosh((\frac{\pi}{2}+\sigma-\xi(\rho,\lambda))\theta)-1}{\theta\sinh(\sigma\theta)}\right\rbrace (t) = O\left(te^{-R^{2}C_{2}^{\sigma,\rho,\lambda}/4t}\right),
    \end{equation} where $C_{2}^{\sigma,\rho,\lambda}>0$ is a constant depending only on $\sigma$, $\rho$ and $\lambda$.
\end{enumerate}
\end{lemma}

\begin{proof}
Let us prove (i). Let $\lbrace A_{N},B_{N} \rbrace_{N\in \mathbb{Z}_{\geq 0}}$ be the covering of $\mathcal{T}$ as given in Lemma \ref{searchlight_coverings}(i). For $(\sigma,\rho,\lambda)\in A_{0}$, the result is immediate by Lemma \ref{inverse_Laplace_transforms}\ref{cosh_sinh_convergent_integral} and for $(\sigma,\rho,\lambda)\in B_{0}$ we see that \begin{equation}
    \frac{\cosh((\frac{\pi}{2}-\sigma+\xi(\rho,\lambda))\theta)-1}{\theta\sinh(\sigma\theta)} = \frac{\tanh(\tfrac{\sigma}{2}\theta)}{\theta}
\end{equation} and the result is immediate from Lemma \ref{inverse_Laplace_transforms}\ref{tanh_remainder}. Let $N\geq 1$. Suppose $(\sigma,\rho,\lambda) \in A_{N}$ and observe that \begin{equation}\begin{split}
    \frac{\cosh((\frac{\pi}{2}-\sigma+\xi(\rho,\lambda))\theta)-1}{\theta\sinh(\sigma\theta)} & = 2  \sum_{n=1}^{N}\frac{\sinh((\tfrac{\pi}{2}-2n\sigma+\xi(\rho,\lambda))\theta)}{\theta} \\ & \enspace \enspace +\frac{\cosh((\frac{\pi}{2}-(2N+1)\sigma + \xi(\rho,\lambda))\theta)-1}{\theta\sinh(\sigma\theta)},
    \end{split}
\end{equation} then the result comes from applying Lemmata \ref{inverse_Laplace_transforms}\ref{cosh_sinh_convergent_integral} and \ref{inverse_Laplace_transforms}\ref{sinh_remainder}. Suppose $(\sigma,\rho,\lambda) \in B_{N}$, then we can observe that \begin{equation}
    \frac{\cosh((\frac{\pi}{2}-\sigma+\xi(\rho,\lambda))\theta)-1}{\theta\sinh(\sigma\theta)} = 2  \sum_{n=1}^{N}\frac{\sinh((\tfrac{\pi}{2}-2n\sigma+\xi(\rho,\lambda))\theta)}{\theta} +\frac{\tanh(\tfrac{\sigma}{2}\theta)}{\theta},
\end{equation} and the result comes from applying Lemmata \ref{inverse_Laplace_transforms}\ref{sinh_remainder} and \ref{inverse_Laplace_transforms}\ref{tanh_remainder}. Since $\lbrace A_{N},B_{N} \rbrace_{N\in \mathbb{Z}_{\geq 0}}$ is a covering of $\mathcal{T}$ we are done.

Now let us prove (ii). Let $\lbrace C_{N},D_{N} \rbrace_{N\in \mathbb{Z}_{\geq 0}}$ be the covering of $\mathcal{T}$ as given in Lemma \ref{searchlight_coverings}(ii). For $(\sigma,\rho,\lambda)\in C_{0}$, the result is immediate by Lemma \ref{inverse_Laplace_transforms}\ref{cosh_sinh_convergent_integral} and for $(\sigma,\rho,\lambda)\in D_{0}$ we see that \begin{equation}
    \frac{\cosh((\frac{\pi}{2}+\sigma-\xi(\rho,\lambda))\theta)-1}{\theta\sinh(\sigma\theta)} = \frac{\tanh(\tfrac{\sigma}{2}\theta)}{\theta}
\end{equation} and the result is immediate from Lemma  \ref{inverse_Laplace_transforms}\ref{tanh_remainder}. Let $N\geq 1$. Suppose $(\sigma,\rho,\lambda) \in C_{N}$ and observe that \begin{equation}\begin{split}
    \frac{\cosh((\frac{\pi}{2}+\sigma-\xi(\rho,\lambda))\theta)-1}{\theta\sinh(\sigma\theta)} & = 2  \sum_{n=1}^{N}\frac{\sinh((\tfrac{\pi}{2}-2(n-1)\sigma-\xi(\rho,\lambda))\theta)}{\theta} \\ & \enspace \enspace +\frac{\cosh((\frac{\pi}{2}-(2N-1)\sigma - \xi(\rho,\lambda))\theta)-1}{\theta\sinh(\sigma\theta)},
    \end{split}
\end{equation} then the result comes from applying Lemmata \ref{inverse_Laplace_transforms}\ref{cosh_sinh_convergent_integral} and \ref{inverse_Laplace_transforms}\ref{sinh_remainder}. Suppose $(\sigma,\rho,\lambda) \in D_{N}$, then we can observe that \begin{equation}
    \frac{\cosh((\frac{\pi}{2}+\sigma-\xi(\rho,\lambda))\theta)-1}{\theta\sinh(\sigma\theta)} = 2  \sum_{n=1}^{N}\frac{\sinh((\tfrac{\pi}{2}-2(n-1)\sigma-\xi(\rho,\lambda))\theta)}{\theta} +\frac{\tanh(\tfrac{\sigma}{2}\theta)}{\theta},
\end{equation} and the result comes from applying Lemmata \ref{inverse_Laplace_transforms}\ref{sinh_remainder} and \ref{inverse_Laplace_transforms}\ref{tanh_remainder}. Since $\lbrace C_{N},D_{N} \rbrace_{N\in \mathbb{Z}_{\geq 0}}$ is a covering of $\mathcal{T}$ we are done.
\end{proof}

\begin{rem}
From the above it is clear that for (i) and (ii), if $\xi(\rho,\lambda)$ depends only on $\rho$, then the constants $C_{1}^{\sigma,\rho,\lambda}>0$ and $C_{2}^{\sigma,\rho,\lambda}>0$ depend only on $\sigma$ and $\rho$. Moreover, in case (i) if $\xi(\rho,\lambda) = 0$ then $C_{1}^{\sigma,\rho,\lambda}>0$ depends only on $\sigma$.
\end{rem}

With this toolbox now in hand we are ready to compute the model computations for the NON and NOON vertex cases. First let us state a result that was computed in \cite[\S 2]{vdBS90} (see equations (2.10) and (2.14) there) and will be used in what follows:
\begin{equation}\label{perimeter_term}
\begin{split}
    & \mathcal{L}^{-1}\left\lbrace \int_{0}^{R} rdr\int_{0}^{\infty} r_{0}dr_{0}\int_{0}^{\infty}d\theta K_{i\theta}(r\sqrt{s})K_{i\theta}(r_{0}\sqrt{s})\frac{2\sinh^{2}(\frac{\pi}{2}\theta)}{\pi^{2}\theta^{2}}\right\rbrace (t) \\
    &  \qquad \qquad = \mathcal{L}^{-1}\left\lbrace \frac{1}{s}\int_{0}^{R} rdr \int_{0}^{\infty}d\theta K_{i\theta}(r\sqrt{s})\frac{\sinh(\frac{\pi}{2}\theta)}{\pi\theta}\right\rbrace (t)\\
    &  \qquad \qquad = \frac{R}{\pi^{1/2}}t^{1/2}  - \frac{R}{\pi^{1/2}}t^{1/2}\int_{1}^{\infty}\frac{dv}{v^{2}}\int_{0}^{1}dy\frac{y}{(1-y^{2})^{1/2}} e^{-R^{2}y^{2}v^{2}/4t}.
\end{split}
\end{equation}

\begin{thm}
\label{NON_wedge} Let $v$ be a NON vertex with interior angle $\gamma$ and exterior angle $\beta$. Under the variable change $\gamma = \rho$, $\beta = \sigma-\rho$, we have that $0<\rho<\sigma < 2\pi$ and the model heat content contribution from $\widetilde{S}_{v}(R)$ is
\begin{equation}
\begin{split}
    & \int_{0}^{R} r \; dr \int_{0}^{\rho} d\phi \int_{0}^{\infty} r_{0} \; dr_{0} \int_{0}^{\rho} d\phi_{0} \; \eta_{W_{\sigma}}(t;r,\phi,r_{0},\phi_{0}) \\ & \qquad \qquad = \frac{1}{2}\rho R^{2} - \frac{R}{\pi^{1/2}}t^{1/2} + \hat{b}(\sigma,\rho)t  \\
    & \qquad \qquad \quad + \frac{R}{\pi^{1/2}}t^{1/2}\int_{1}^{\infty}\frac{dv}{v^{2}}\int_{0}^{1}dy\frac{y}{(1-y^{2})^{1/2}} e^{-R^{2}y^{2}v^{2}/4t} + O\left(te^{-R^{2}C_{\sigma,\rho}/4t}\right) \\
    & \qquad \qquad = |\widetilde{S}_{v}(R)| - \frac{R}{\pi^{1/2}}t^{1/2} + b(\gamma,\beta)t  \\
    & \qquad \qquad \quad + \frac{R}{\pi^{1/2}}t^{1/2}\int_{1}^{\infty}\frac{dv}{v^{2}}\int_{0}^{1}dy\frac{y}{(1-y^{2})^{1/2}} e^{-R^{2}y^{2}v^{2}/4t} + O\left(te^{-R^{2}C_{\sigma,\rho}/4t}\right),
\end{split}
\end{equation} where \begin{equation}
    \hat{b}(\sigma,\rho) = \int_{0}^{\infty}d\theta \frac{\cosh\left(\frac{\pi}{2}\theta\right)\cosh((\sigma-2\rho)\theta)-\cosh\left(\left(\frac{\pi}{2}-\sigma\right)\theta\right)}{2\sinh(\sigma\theta)\sinh(\frac{\pi}{2}\theta)}
\end{equation} and $C_{\sigma,\rho},C_{\gamma,\beta}>0$ are constants depending only on $\sigma$ and $\rho$ respectively.
\end{thm}

\begin{proof}
As in the case of the NN wedge, we compute the angular terms first. Again we can do this by Fubini's theorem. With standard hyperbolic trigonometric identities, one can show that \begin{equation}
\begin{split}
    & \int_{0}^{\rho}d\phi \int_{0}^{\rho} d\phi_{0} \; \Phi_{\sigma}(\theta,\phi,\phi_{0}) \\
    & \qquad = \frac{2\rho}{\theta}\sinh(\pi\theta) + \frac{2}{\theta^{2}}(\cosh((\pi-\rho)\theta) - \cosh(\pi \theta)) +\frac{2\sinh((\pi-\sigma)\theta)}{\theta^{2}\sinh(\sigma\theta)} \left(\cosh(\rho\theta)-1\right) \\
    & \qquad \qquad + \frac{\sinh(\pi\theta)}{\theta^{2}\sinh(\sigma\theta)} \left(\cosh((\sigma-2\rho)\theta)+\cosh(\sigma\theta)-2\cosh((\sigma-\rho)\theta)\right) \\
    & \qquad = \frac{2\rho}{\theta}\sinh(\pi\theta) - \frac{2}{\theta^{2}}\sinh^{2}(\tfrac{\pi}{2}\theta) + \frac{2\sinh(\frac{\pi}{2}\theta)}{\theta^{2}\sinh(\sigma\theta)} \Big\lbrace \cosh\left(\tfrac{\pi}{2}\theta\right)\cosh((\sigma-2\rho)\theta) \\ & \qquad \qquad -\cosh\left(\left(\tfrac{\pi}{2}-\sigma\right)\theta\right) \Big\rbrace.
\end{split}
\end{equation} We know how to treat the first two terms in the last line of the equation above from equations \eqref{area_term} and \eqref{perimeter_term}, so we only need to treat the third term. Applying the identity in equation \eqref{bessel_trick} twice, we see that \begin{equation}
\begin{split}
& \mathcal{L}^{-1}\bigg\lbrace 2\int_{0}^{\infty} r dr \int_{0}^{\infty}r_{0}dr_{0}\int_{0}^{\infty}d\theta K_{i\theta}(r\sqrt{s})K_{i\theta}(r_{0}\sqrt{s}) \sinh(\tfrac{\pi}{2}\theta) \\ & \qquad \times \frac{\cosh\left(\frac{\pi}{2}\theta\right)\cosh((\sigma-2\rho)\theta)-\cosh\left(\left(\frac{\pi}{2}-\sigma\right)\theta\right)}{\pi^{2}\theta^{2}\sinh(\sigma\theta)} \bigg\rbrace(t)  \\
& \qquad \qquad = \mathcal{L}^{-1}\left\lbrace\frac{1}{2s^2}\int_{0}^{\infty}d\theta \frac{\cosh\left(\frac{\pi}{2}\theta\right)\cosh((\sigma-2\rho)\theta)-\cosh\left(\left(\frac{\pi}{2}-\sigma\right)\theta\right)}{\sinh(\sigma\theta)\sinh(\tfrac{\pi}{2}\theta)} \right\rbrace \\
& \qquad\qquad = \frac{t}{2}\int_{0}^{\infty}d\theta \frac{\cosh\left(\frac{\pi}{2}\theta\right)\cosh((\sigma-2\rho)\theta)-\cosh\left(\left(\frac{\pi}{2}-\sigma\right)\theta\right)}{\sinh(\sigma\theta)\sinh(\tfrac{\pi}{2}\theta)}.
\end{split}
\end{equation} Thus, by linearity of the inverse Laplace transform, we have that  \begin{equation}
\begin{split}
    & \int_{0}^{R} r \; dr \int_{0}^{\rho} d\phi \int_{0}^{\infty} r_{0} \; dr_{0} \int_{0}^{\rho} d\phi_{0} \; \eta_{W_{\sigma}}(t,r,\phi,r_{0},\phi_{0}) \\
    & \qquad =  \frac{1}{2}\rho R^{2} - \frac{R}{\pi^{1/2}}t^{1/2} + \hat{b}(\sigma,\rho)t -S_{1}(t) + \frac{R}{\pi^{1/2}}t^{1/2}\int_{1}^{\infty}\frac{dv}{v^{2}}\int_{0}^{1}dy\frac{y}{(1-y^{2})^{1/2}} e^{-R^{2}y^{2}v^{2}/4t},
\end{split}
\end{equation} where \begin{equation}
S_{1}(t) = \mathcal{L}^{-1}\left\lbrace \frac{1}{s} \int_{R}^{\infty} r  dr\int_{0}^{\infty}d\theta K_{i\theta}(r\sqrt{s})\frac{\cosh\left(\frac{\pi}{2}\theta\right)\cosh((\sigma-2\rho)\theta)-\cosh\left(\left(\frac{\pi}{2}-\sigma\right)\theta\right)}{\pi\theta\sinh(\sigma\theta)} \right\rbrace(t).
\end{equation} Thus, it suffices to show that $S_{1}(t)$ is exponentially small as $t\downarrow 0$. Observe that \begin{equation}
\begin{split}
& \frac{\cosh\left(\frac{\pi}{2}\theta\right)\cosh((\sigma-2\rho)\theta)-\cosh\left(\left(\frac{\pi}{2}-\sigma\right)\theta\right)}{\pi\theta\sinh(\sigma\theta)} \\
& \qquad \qquad = \frac{\cosh((\tfrac{\pi}{2}-\sigma+2\rho)\theta)+\cosh((\tfrac{\pi}{2}+\sigma-2\rho)\theta)-2\cosh((\tfrac{\pi}{2}-\sigma)\theta)}{2\pi\theta\sinh(\sigma\theta)}  \\
& \qquad \qquad = \frac{\cosh((\tfrac{\pi}{2}-\sigma+2\rho)\theta)-1}{2\pi\theta\sinh(\sigma\theta)}  + \frac{\cosh((\tfrac{\pi}{2}+\sigma-2\rho)\theta)-1}{2\pi\theta\sinh(\sigma\theta)}  \\ & \qquad \qquad \qquad - \frac{\cosh((\tfrac{\pi}{2}-\sigma)\theta)-1}{\pi\theta\sinh(\sigma\theta)}.
\end{split}
\end{equation} Hence, we have that \begin{equation}
    \begin{split}
        S_{1}(t) & = \mathcal{L}^{-1}\left\lbrace \frac{1}{s} \int_{R}^{\infty} r dr\int_{0}^{\infty}d\theta K_{i\theta}(r\sqrt{s})\frac{\cosh((\tfrac{\pi}{2}-\sigma+2\rho)\theta)-1}{2\pi\theta\sinh(\sigma\theta)} \right\rbrace(t) \\
        & \qquad + \mathcal{L}^{-1}\left\lbrace \frac{1}{s} \int_{R}^{\infty} r dr\int_{0}^{\infty}d\theta K_{i\theta}(r\sqrt{s})\frac{\cosh((\tfrac{\pi}{2}+\sigma-2\rho)\theta)-1}{2\pi\theta\sinh(\sigma\theta)} \right\rbrace(t) \\
        & \qquad -\mathcal{L}^{-1}\left\lbrace \frac{1}{s} \int_{R}^{\infty} r dr\int_{0}^{\infty}d\theta K_{i\theta}(r\sqrt{s})\frac{\cosh((\tfrac{\pi}{2}-\sigma)\theta)-1}{\pi\theta\sinh(\sigma\theta)} \right\rbrace(t).
    \end{split}
\end{equation} Using Lemmata \ref{angle_remainder}(i) and \ref{angle_remainder}(ii), we see immediately that $S_{1}(t) = O\left(te^{-R^{2}C_{\sigma,\rho}/t}\right)$ where $C_{\sigma,\rho}>0$ is a constant depending only on $\sigma$ and $\rho$. Undoing the variable substitution one sees that, $\hat{b}(\sigma,\rho)=b(\gamma,\beta)$ with $b$ as defined in  Theorem \ref{thm:main} and $\frac{1}{2}\rho R^{2} = \frac{1}{2}\gamma R^{2} = |\widetilde{S}_{v}(R)|$, which concludes the proof.
\end{proof}

\begin{thm}
\label{NOON_wedge} Let $v$ be a NOON vertex with middle angle $\gamma$ and exterior angles $\beta$ and $\alpha$. Under the variable change $\gamma = \rho - \lambda$, $\beta = \sigma-\rho$ and $\alpha = \lambda$, we have that $0<\lambda<\rho<\sigma < 2\pi$. The model heat content contribution from $\widetilde{S}_{v}(R)$ when $\gamma$ is an interior angle of $\widetilde{D}$ is
\begin{equation}
\begin{split}
    &\int_{0}^{R} r \; dr\int_{\lambda}^{\rho}d\phi \int_{0}^{\infty}r_{0} \; dr_{0}\int_{\lambda}^{\rho} d\phi_{0} \; \eta_{W_{\sigma}}(t;r,\phi,r_{0},\phi_{0}) \\
    & \qquad \qquad = \frac{1}{2}(\rho-\lambda)R^{2} - \frac{2R}{\pi^{1/2}}t^{1/2}+\hat{c}(\sigma,\rho,\lambda)t  \\
    & \qquad \qquad \quad + \frac{2R}{\pi^{1/2}}t^{1/2}\int_{1}^{\infty}\frac{dv}{v^{2}}\int_{0}^{1}dy\frac{y}{(1-y^{2})^{1/2}} e^{-R^{2}y^{2}v^{2}/4t} + O\left(te^{-R^{2}C_{\sigma,\rho,\lambda}/4t}\right) \\
    & \qquad \qquad = |\widetilde{S}_{v}(R)| - \frac{2R}{\pi^{1/2}}t^{1/2} + c(\gamma,\beta,\alpha)t \\
    & \qquad \qquad \quad + \frac{2R}{\pi^{1/2}}t^{1/2}\int_{1}^{\infty}\frac{dv}{v^{2}}\int_{0}^{1}dy\frac{y}{(1-y^{2})^{1/2}} e^{-R^{2}y^{2}v^{2}/4t} + O\left(te^{-R^{2}C_{\gamma,\beta,\alpha}/4t}\right),
\end{split}
\end{equation} and when $\beta$ and $\alpha$ are interior angles of $\widetilde{D}$ is
\begin{equation}
\begin{split}
    &\int_{0}^{R} r \; dr\int_{(0,\lambda)\cup(\rho,\sigma)}d\phi \int_{0}^{\infty}r_{0} \; dr_{0}\int_{(0,\lambda)\cup(\rho,\sigma)} d\phi_{0} \; \eta_{W_{\sigma}}(t;r,\phi,r_{0},\phi_{0}) \\
    & \qquad \qquad = \frac{1}{2}(\lambda+\sigma-\rho)R^{2} - \frac{2R}{\pi^{1/2}}t^{1/2}+\hat{c}(\sigma,\rho,\lambda)t  \\
    & \qquad \qquad \quad + \frac{2R}{\pi^{1/2}}t^{1/2}\int_{1}^{\infty}\frac{dv}{v^{2}}\int_{0}^{1}dy\frac{y}{(1-y^{2})^{1/2}} e^{-R^{2}y^{2}v^{2}/4t} + O\left(te^{-R^{2}C_{\sigma,\rho,\lambda}/4t}\right) \\
    & \qquad \qquad = |\widetilde{S}_{v}(R)| - \frac{2R}{\pi^{1/2}}t^{1/2} + c(\gamma,\beta,\alpha)t \\
    & \qquad \qquad \quad + \frac{2R}{\pi^{1/2}}t^{1/2}\int_{1}^{\infty}\frac{dv}{v^{2}}\int_{0}^{1}dy\frac{y}{(1-y^{2})^{1/2}} e^{-R^{2}y^{2}v^{2}/4t} + O\left(te^{-R^{2}C_{\gamma,\beta,\alpha}/4t}\right).
\end{split}
\end{equation}

Here, \begin{equation}
\begin{split}
    \hat{c}(\sigma,\rho,\lambda) & = \hat{b}(\sigma,\rho) + \hat{b}(\sigma,\lambda)  \\ & \qquad - \int_{0}^{\infty}d\theta \frac{\cosh\left(\frac{\pi}{2}\theta\right)\left(\cosh((\sigma-\rho-\lambda)\theta)-\cosh((\sigma-\rho+\lambda)\theta)\right)}{\sinh(\sigma\theta)\sinh(\frac{\pi}{2}\theta)}
\end{split}
\end{equation} and $C_{\sigma,\rho,\lambda},C_{\gamma,\beta,\alpha}>0$ are constants depending only on $\sigma,\rho,\lambda>0$ and $\gamma,\beta,\alpha>0$ respectively.
\end{thm}

\begin{proof}
First let us consider the case where $\gamma$ is an interior angle of $\widetilde{D}$. Again, as in the case of the NN and NON wedges, we begin with the integrals over the angles. We first observe that by Fubini's theorem \begin{equation}
\begin{split}
    \int_{\lambda}^{\rho}d\phi \int_{\lambda}^{\rho}d\phi_{0} \; \Phi_{\sigma}(\theta,\phi,\phi_{0}) & = \int_{0}^{\rho}d\phi\int_{0}^{\rho}d\phi_{0} \; \Phi_{\sigma}(\theta,\phi,\phi_{0}) +\int_{0}^{\lambda}d\phi\int_{0}^{\lambda}d\phi_{0}\; \Phi_{\sigma}(\theta,\phi,\phi_{0})  \\ & \qquad -2\int_{0}^{\lambda}d\phi\int_{0}^{\rho}d\phi_{0} \; \Phi_{\sigma}(\theta,\phi,\phi_{0})
\end{split}
\end{equation} and we have treated the first two terms on the right-hand side in the NON wedge case. So only the third term needs to be treated. We have that \begin{equation}\label{NOONAngle1}
    \begin{split}
        &\int_{0}^{\lambda}d\phi\int_{0}^{\rho}d\phi_{0}\cosh((\pi-|\phi-\phi_{0}|)\theta) \\ & \qquad  = \frac{2\lambda}{\theta}\sinh(\pi\theta) + \frac{1}{\theta^{2}}\big[\cosh((\pi-\lambda)\theta) +\cosh((\pi-\rho)\theta) -\cosh(\pi\theta) \\
        & \qquad \qquad -\cosh((\pi+\lambda-\rho)\theta)\big]
        \\ & \qquad = \frac{2\lambda}{\theta}\sinh(\pi\theta) + \frac{1}{2\theta^{2}\sinh(\sigma\theta)}\big[\sinh((\pi+\sigma-\lambda)\theta) -\sinh((\pi-\sigma-\lambda)\theta)  \\& \qquad \qquad + \sinh((\pi+\sigma-\rho)\theta) - \sinh((\pi-\sigma-\rho)\theta) - \sinh((\pi+\sigma)\theta) \\ & \qquad \qquad +\sinh((\pi-\sigma)\theta) -\sinh((\pi+\sigma-\rho+\lambda)\theta)+\sinh((\pi-\sigma-\rho+\lambda)\theta) \big],
    \end{split}
\end{equation} as well as \begin{equation}\label{NOONAngle2}
    \begin{split}
        & \frac{\sinh(\pi\theta)}{\sinh(\sigma\theta)} \int_{0}^{\lambda}d\phi \int_{0}^{\rho}d\phi_{0} \cosh((\phi+\phi_{0}-\sigma)\theta) \\ & \enspace \enspace = \frac{\sinh(\pi\theta)}{\theta^{2}\sinh(\sigma\theta)} \big[\cosh((\sigma-\rho-\lambda)\theta) + \cosh(\sigma\theta) - \cosh((\sigma-\lambda)\theta) \\ & \enspace\enspace\enspace\enspace -\cosh((\sigma-\rho)\theta) \big] \\ & \enspace \enspace = \frac{1}{2\theta^{2}\sinh(\sigma\theta)} \big[2\sinh(\pi\theta)\cosh((\sigma-\rho-\lambda)\theta) + \sinh((\pi+\sigma)\theta) \\ & \enspace\enspace\enspace\enspace + \sinh((\pi-\sigma)\theta) -\sinh((\pi+\sigma-\lambda)\theta)-\sinh((\pi-\sigma+\lambda)\theta) \\ &
        \enspace\enspace\enspace\enspace -\sinh((\pi+\sigma-\rho)\theta)-\sinh((\pi-\sigma+\rho)\theta) \big],
    \end{split}
\end{equation} and that \begin{equation}\label{NOONAngle3}
    \begin{split}
        & \frac{\sinh((\pi-\sigma)\theta)}{\sinh(\sigma\theta)} \int_{0}^{\lambda}d\phi \int_{0}^{\rho}d\phi_{0} \cosh((\phi-\phi_{0})\theta) \\ & \enspace \enspace = \frac{\sinh((\pi-\sigma)\theta)}{\theta^{2}\sinh(\sigma\theta)} \big[\cosh(\lambda\theta)+\cosh(\rho\theta) - \cosh((\rho-\lambda)\theta)-1 \big] \\ & \enspace \enspace = \frac{1}{2\theta^{2}\sinh(\sigma\theta)} \big[ \sinh((\pi-\sigma+\lambda)\theta) + \sinh((\pi-\sigma-\lambda)\theta)  \\
        & \enspace\enspace\enspace\enspace +\sinh((\pi - \sigma + \rho)\theta) + \sinh((\pi - \sigma - \rho)\theta)
        -\sinh((\pi-\sigma+\rho-\lambda)\theta)\\
        & \enspace\enspace\enspace\enspace  - \sinh((\pi-\sigma-\rho+\lambda)\theta) - 2\sinh((\pi-\sigma)\theta)\big].
    \end{split}
\end{equation} Summing \eqref{NOONAngle1}, \eqref{NOONAngle2}, and \eqref{NOONAngle3} we obtain that \begin{equation}\label{eqn:NOON_Sum}
\begin{split}
    \int_{0}^{\lambda}d\phi \int_{0}^{\rho} d\phi_{0} \; \Phi_{\sigma}(\theta,\phi,\phi_{0}) & = \frac{2\lambda}{\theta}\sinh(\pi\theta) +\frac{1}{\theta^{2}\sinh(\sigma\theta)}\bigg(\sinh(\pi\theta)\cosh((\sigma-\rho-\lambda)\theta) \\
    & \enspace \enspace -\frac{1}{2}\sinh((\pi +\sigma-\rho+\lambda)\theta)-\frac{1}{2}\sinh((\pi -\sigma+\rho-\lambda)\theta)\bigg) \\
    & = \frac{2\lambda}{\theta}\sinh(\pi\theta) +\frac{\sinh(\pi\theta)}{\theta^{2}\sinh(\sigma\theta)}\Big(\cosh((\sigma-\rho-\lambda)\theta) \\
    & \enspace \enspace -\cosh((\sigma-\rho+\lambda)\theta)\Big).
\end{split}
\end{equation} Hence, using the identity \eqref{area_term} and the identity \eqref{bessel_trick} twice, we have that \begin{equation}
\begin{split}
    & \mathcal{L}^{-1}\left\lbrace\int_{0}^{R}r \; dr\int_{0}^{\infty}r_{0} \; dr_{0} \int_{0}^{\lambda}d\phi\int_{0}^{\rho}d\phi_{0}\; G_{W_{\sigma}}(s,r,\phi,r_{0},\phi_{0})\right\rbrace (t) \\
    & \enspace \enspace = \frac{1}{2}\lambda R^{2} - S_{2}(t)+ t \int_{0}^{\infty}d\theta\frac{\cosh(\tfrac{\pi}{2}\theta)\left(\cosh((\sigma-\rho-\lambda)\theta)-\cosh((\sigma-\rho+\lambda)\theta)\right)}{2\sinh(\sigma\theta)\sinh(\tfrac{\pi}{2}\theta)}
\end{split}
\end{equation} where \begin{equation}
\begin{split}
S_{2}(t) = \mathcal{L}^{-1}\bigg\lbrace & \frac{1}{s} \int_{R}^{\infty} r \; dr\int_{0}^{\infty}d\theta \; K_{i\theta}(r\sqrt{s}) \\ & \times \frac{\cosh\left(\frac{\pi}{2}\theta\right)\left(\cosh((\sigma-\rho-\lambda)\theta)-\cosh((\sigma-\rho+\lambda)\theta)\right)}{\pi\theta\sinh(\sigma\theta)} \bigg\rbrace(t).
\end{split}
\end{equation} Now observing that \begin{equation}
    \begin{split}
        & \frac{\cosh\left(\frac{\pi}{2}\theta\right)\cosh((\sigma-\rho-\lambda)\theta)}{\pi\theta\sinh(\sigma\theta)}-\frac{\cosh\left(\frac{\pi}{2}\theta\right)\cosh((\sigma-\rho+\lambda)\theta)}{\pi\theta\sinh(\sigma\theta)}  \\
        & \qquad \qquad \qquad = \frac{\cosh((\tfrac{\pi}{2}+\sigma-\rho-\lambda)\theta)-1}{2\pi\theta \sinh(\sigma\theta)} + \frac{\cosh((\tfrac{\pi}{2}-\sigma+\rho+\lambda)\theta)-1}{2\pi\theta \sinh(\sigma\theta)} \\
        & \qquad \qquad \qquad \quad - \frac{\cosh((\tfrac{\pi}{2}+\sigma-\rho+\lambda)\theta)-1}{2\pi\theta \sinh(\sigma\theta)}  - \frac{\cosh((\tfrac{\pi}{2}-\sigma+\rho-\lambda)\theta)-1}{2\pi\theta \sinh(\sigma\theta)}
    \end{split}
\end{equation} by using Lemmata \ref{angle_remainder}(i) and \ref{angle_remainder}(ii) we have $S_{2}(t) = O\left(te^{-R^2C_{\sigma,\rho,\lambda}/t}\right)$ where $C_{\sigma,\rho,\lambda}>0$ is a constant depending only on $\sigma$, $\rho$ and $\lambda$. Now by the linearity of inverse Laplace transform we have that \begin{equation}
\begin{split}
    & \mathcal{L}^{-1}\left\lbrace \int_{0}^{R} r dr\int_{\lambda}^{\rho}d\phi \int_{0}^{\infty}r_{0}dr_{0}\int_{\lambda}^{\rho} d\phi_{0} G_{W_{\sigma}}(s,r,\phi,r_{0},\phi_{0})\right\rbrace(t) \\
    & \enspace \enspace = \mathcal{L}^{-1}\left\lbrace \int_{0}^{R} r dr\int_{0}^{\rho}d\phi \int_{0}^{\infty}r_{0}dr_{0}\int_{0}^{\rho} d\phi_{0} G_{W_{\sigma}}(s,r,\phi,r_{0},\phi_{0})\right\rbrace(t) \\
    & \enspace \enspace \enspace \enspace + \mathcal{L}^{-1}\left\lbrace \int_{0}^{R} r dr\int_{0}^{\lambda}d\phi \int_{0}^{\infty}r_{0}dr_{0}\int_{0}^{\lambda} d\phi_{0} G_{W_{\sigma}}(s,r,\phi,r_{0},\phi_{0})\right\rbrace(t)\\
    & \enspace \enspace \enspace \enspace -2 \mathcal{L}^{-1}\left\lbrace \int_{0}^{R} r dr\int_{0}^{\lambda}d\phi \int_{0}^{\infty}r_{0}dr_{0}\int_{0}^{\rho} d\phi_{0} G_{W_{\sigma}}(s,r,\phi,r_{0},\phi_{0})\right\rbrace(t) \\
    & \enspace \enspace = \frac{1}{2}(\rho-\lambda)R^{2} - \frac{2R}{\pi^{1/2}}t^{1/2}+\hat{c}(\sigma,\rho,\lambda)t \\
    & \enspace \enspace \enspace \enspace + \frac{2R}{\pi^{1/2}}t^{1/2}\int_{1}^{\infty}\frac{dv}{v^{2}}\int_{0}^{1}dy\frac{y}{(1-y^{2})^{1/2}} e^{-R^{2}y^{2}v^{2}/4t} + O\left(te^{-R^{2}C_{\sigma,\rho,\lambda}/4t}\right).
\end{split}
\end{equation} Undoing the variable substitution one sees that, $\hat{c}(\sigma,\rho,\lambda)=c(\gamma,\beta,\alpha)$ with $c$ as defined in  Theorem \ref{thm:main} and $\frac{1}{2}(\rho-\lambda)R^{2} = \frac{1}{2}\gamma R^{2} = |\widetilde{S}_{v}(R)|$.

The case where $\beta$ and $\alpha$ are interior angles of $\widetilde{D}$ follows on from this. Note that by Fubini's theorem we have that \begin{equation}
\begin{split}
    & \int_{(0,\lambda)\cup(\rho,\sigma)}d\phi \int_{(0,\lambda)\cup(\rho,\sigma)} d\phi_{0} \; \Phi_{\sigma}(\theta,\phi,\phi_{0}) \\
    & \qquad \qquad = \int_{0}^{\sigma}d\phi \int_{0}^{\sigma} d\phi_{0}\; \Phi_{\sigma}(\theta,\phi,\phi_{0}) + \int_{\lambda}^{\rho}d\phi \int_{\lambda}^{\rho} d\phi_{0}\; \Phi_{\sigma}(\theta,\phi,\phi_{0}) \\
    & \qquad \qquad \quad - 2 \int_{\lambda}^{\rho}d\phi \int_{0}^{\sigma} d\phi_{0}\; \Phi_{\sigma}(\theta,\phi,\phi_{0}).
\end{split}
\end{equation} We see from \eqref{eqn:NOON_Sum} that
\begin{equation}
    \int_{\lambda}^{\rho}d\phi \int_{0}^{\sigma} d\phi_{0}\; \Phi_{\sigma}(\theta,\phi,\phi_{0}) = \frac{2(\rho-\lambda)}{\theta}\sinh(\pi\theta).
\end{equation}
Hence by \eqref{area_term} we have \begin{equation}
    \mathcal{L}^{-1}\left\lbrace 2 \int_{0}^{R} r dr\int_{\lambda}^{\rho}d\phi \int_{0}^{\infty}r_{0}dr_{0}\int_{0}^{\sigma} d\phi_{0} G_{W_{\sigma}}(s,r,\phi,r_{0},\phi_{0})\right\rbrace(t) = (\rho-\lambda)R^{2},
\end{equation}
and by Lemma \ref{lem:NN} we have
\begin{equation}
    \mathcal{L}^{-1}\left\lbrace \int_{0}^{R} r dr\int_{0}^{\sigma}d\phi \int_{0}^{\infty}r_{0}dr_{0}\int_{0}^{\sigma} d\phi_{0} G_{W_{\sigma}}(s,r,\phi,r_{0},\phi_{0})\right\rbrace(t) = \frac{\sigma}{2}R^{2}.
\end{equation}
Thus, we have \begin{equation}
    \begin{split}
        & \mathcal{L}^{-1}\left\lbrace \int_{0}^{R} r \; dr\int_{(0,\lambda)\cup(\rho,\sigma)}d\phi \int_{0}^{\infty}r_{0} \; dr_{0}\int_{(0,\lambda)\cup(\rho,\sigma)} d\phi_{0} G_{W_{\sigma}}(s,r,\phi,r_{0},\phi_{0})\right\rbrace(t) \\
        & \qquad \qquad = \frac{1}{2}(\sigma + \rho - \lambda - 2(\rho-\lambda))R^{2} - \frac{2R}{\pi^{1/2}}t^{1/2}+\hat{c}(\sigma,\rho,\lambda)t \\
    & \qquad \qquad \quad + \frac{2R}{\pi^{1/2}}t^{1/2}\int_{1}^{\infty}\frac{dv}{v^{2}}\int_{0}^{1}dy\frac{y}{(1-y^{2})^{1/2}} e^{-R^{2}y^{2}v^{2}/4t} + O\left(te^{-R^{2}C_{\sigma,\rho,\lambda}/4t}\right) \\
    & \qquad \qquad = \frac{1}{2}(\beta+\alpha)R^{2} - \frac{2R}{\pi^{1/2}}t^{1/2}+c(\gamma,\beta,\alpha)t \\
    & \qquad \qquad \quad + \frac{2R}{\pi^{1/2}}t^{1/2}\int_{1}^{\infty}\frac{dv}{v^{2}}\int_{0}^{1}dy\frac{y}{(1-y^{2})^{1/2}} e^{-R^{2}y^{2}v^{2}/4t} + O\left(te^{-R^{2}C_{\gamma,\beta,\alpha}/4t}\right)
    \end{split}
\end{equation} which completes the proof.

\end{proof}

\section{Comparisons for model computations}

We now show that the difference between our model heat content contributions and the actual heat content contributions are exponentially small as $t\downarrow 0$.
Although some comparison results for Neumann heat kernels are known \cite{NRS20}, these were not sufficient for our purposes. In addition, the Neumann heat kernel does not satisfy all properties that the Dirichlet heat kernel does.
For example, domain monotonicity for the Neumann heat kernel holds in some cases, see for example \cite{CZ94,C86,H94,K89}, but does not hold in general \cite{BB93}.

We obtain the required results using probabilistic methods as has been done before in \cite{vdBS90} for Dirichlet boundary conditions. The probabilistic interpretation arises from the fact that transition densities for Brownian motion (BM) are given by heat kernels. We note that for us BM will be associated with the operator $-\Delta + \partial_{t}$, which is merely the `standard' BM in the probability literature run at twice the speed.
Moreover, imposing Neumann boundary conditions corresponds to reflecting BM at the boundary, which we call reflecting Brownian motion (RBM). A good summary of RBM's and their relation to Neumann heat kernels for $C^{3}$ domains is given in \cite{H84}. We use the construction of RBM's in polygonal domains given by Gallavotti and McKean in \cite{GM72}, which we summarise below for convenience of the reader.

\begin{figure}
    \centering
    \begin{tikzpicture}
    \draw[->, line width=0.5mm,yshift=0.866cm,xshift=-4.5cm] (0,0) -- (2,0);
    \draw [xshift = -7cm] (0,0) -- (2,0) -- (1,1.732) -- cycle;
    \node[xshift=-7cm] at (1,0.577) {$D^{*}$};
    \draw (0,0) -- (2,0) -- (1,1.732) -- (-1,1.732) -- (-2,0) -- (-1, -1.732) -- (1,-1.732) -- (2,0);
    \draw (0,0) -- (1,1.732);
    \draw (0,0) -- (-1,1.732);
    \draw (0,0) -- (-2,0);
    \draw (0,0) -- (-1,-1.732);
    \draw (0,0) -- (1,-1.732);
    \draw[xshift=2cm] (0,0) -- (1,1.732);
    \draw[xshift=2cm] (0,0) -- (2,0);
    \draw[xshift=2cm] (0,0) -- (1,-1.732);
    \draw[xshift = 2cm] (1,-1.732) -- (2,0) -- (1,1.732) -- (-1,1.732);
    \draw[xshift = 4cm] (1,-1.732) -- (2,0) -- (0,0);
    \draw[xshift = 5cm, yshift = -1.732cm] (-1,1.732) -- (0,0) -- (-2,0);
    \draw[xshift = 1cm, yshift = 1.732cm] (0,0) -- (1,1.732) -- (2,0);
    \node at (1,0.577) {$D^{*}$};
    \node at (0,1) {$D^{*}_{1}$};
    \node at (-1,0.577) {$D^{*}_{12}$};
    \node at (-1,-0.577) {$D^{*}_{121}$};
    \node at (0,-1.2) {$D^{*}_{1212}$};
    \node at (1,-0.4) {$D^{*}_{12121}$};
    \node[xshift=2cm] at (0,1) {$D^{*}_{3}$};
    \node[xshift=2cm] at (1,0.577) {$D^{*}_{32}$};
    \node[xshift=2cm] at (1,-0.5) {$D^{*}_{323}$};
    \node[xshift=4cm] at (0,-1.2) {$D^{*}_{3231}$};
    \node[xshift=4cm] at (1,-0.4) {$D^{*}_{32313}$};
    \node[xshift=1cm,yshift=1.732cm] at (1,0.577) {$D^{*}_{31}$};
    \filldraw[black,xshift=-7cm] (0,0) circle (0.1cm);
    \filldraw[black,xshift=-7cm] (2,0) circle (0.1cm);
    \filldraw[black,xshift=-7cm] (1,1.732) circle (0.1cm);
    \filldraw[black] (0,0) circle (0.1cm);
    \filldraw[black] (1,1.732) circle (0.1cm);
    \filldraw[black] (-1,1.732) circle (0.1cm);
    \filldraw[black] (-2,0) circle (0.1cm);
    \filldraw[black] (-1,-1.732) circle (0.1cm);
    \filldraw[black] (1,-1.732) circle (0.1cm);
    \filldraw[black] (2,0) circle (0.1cm);
    \filldraw[black,xshift=4cm] (-2,0) circle (0.1cm);
    \filldraw[black,xshift=4cm] (-1,1.732) circle (0.1cm);
    \filldraw[black,xshift=4cm] (-1,-1.732) circle (0.1cm);
    \filldraw[black,xshift=4cm] (0,0) circle (0.1cm);
    \filldraw[black,xshift=4cm] (1,-1.732) circle (0.1cm);
    \filldraw[black,xshift=4cm] (2,0) circle (0.1cm);
    \filldraw[black,xshift=1cm,yshift=1.732cm] (1,1.732) circle (0.1cm);
    \filldraw[white,xshift=-7cm] (1,0) circle (0.2cm) node {\color{black!50} 2};
    \filldraw[white,xshift=-7cm] (1.5,0.866) circle (0.2cm) node {\color{black!50} 3};
    \filldraw[white,xshift=-7cm] (0.5,0.866) circle (0.2cm) node {\color{black!50} 1};
    \filldraw[white] (1,0) circle (0.2cm) node {\color{black!50} 2};
    \filldraw[white] (1.5,0.866) circle (0.2cm) node {\color{black!50} 3};
    \filldraw[white] (0.5,0.866) circle (0.2cm) node {\color{black!50} 1};
    \filldraw[white] (-0.5,0.866) circle (0.2cm) node {\color{black!50} 2};
    \filldraw[white] (-1,0) circle (0.2cm) node {\color{black!50} 1};
    \filldraw[white] (-0.5,-0.866) circle (0.2cm) node {\color{black!50} 2};
    \filldraw[white] (0.5,-0.866) circle (0.2cm) node {\color{black!50} 1};
    \filldraw[white,xshift=2cm] (0,1.732) circle (0.2cm) node {\color{black!50} 1};
    \filldraw[white,xshift=2cm] (0.5,0.866) circle (0.2cm) node {\color{black!50} 2};
    \filldraw[white,xshift=2cm] (1,0) circle (0.2cm) node {\color{black!50} 3};
    \filldraw[white,xshift=4cm] (-0.5,-0.866) circle (0.2cm) node {\color{black!50} 1};
    \filldraw[white,xshift=4cm] (0.5,-0.866) circle (0.2cm) node {\color{black!50} 3};
    \node[xshift=4cm,yshift=1.732cm] at (0,0) {$M$};
    \end{tikzpicture}
    \caption{Visualisation of part of the manifold generated by an equilateral triangle $D^{*}$. Note that conventionally we would identify $D^{*}_{12121}$ with $D^{*}_{2}$ but we treat them as distinct here.
    (This figure is an adaptation of Figures 2 and 3 in \cite{GM72}.)}
    \label{fig:manifold}
\end{figure}

Let $D$ be a, possibly unbounded, polygonal domain whose edges are labelled $1$ through $n$. Then for each string $a_{1}a_{2}\cdots a_{n}$ with $1\leq a_{i} \leq n$ and $a_{i+1}\neq a_{i}$ we can obtain a copy $\overline{D}_{a_{1}a_{2}\cdots a_{n}}$ of $\overline{D}$ by reflecting $\overline{D}$ across the sides $a_{1},\; a_{2}, \; \ldots ,a_{n}$ successively. We also have the copy of $\overline{D}$ obtained by carrying out no reflection and simply denote it $\overline{D}$. We declare each of these copies of $\overline{D}$ to be different and the collection $K$ of all these copies is a covering sheet of $\overline{D}$. From $K$, we can obtain the open manifold $M:= K - \lbrace \text{images of vertices of $D$}\rbrace$, which we view as a flat Riemannian manifold and call it the manifold generated by $D$. Figure \ref{fig:manifold} shows a visualisation of this for an equilateral triangle. For $M$, there is the self-evident continuous projection $\Psi_{M} : M \to D^{*} := (D\cup \partial D)\backslash V$, where $V$ denotes the vertices of $D$.
We obtain an RBM $X$ on $D$ by simply projecting a BM $B$ on $M$ onto $D^{*}$ via $\Psi_{M}$, i.e. $X = \Psi_{M} \circ B$.

Before proving our comparisons, let us briefly give some notation: $\mathbb{P}_{x}^{\Omega}$ is the probability measure associated with BM on $\Omega$ if $\Omega$ is a manifold and RBM on $\Omega$ if $\Omega$ is a polygonal domain; $X$ and $Y$ will denote RBM's on some given polygonal domains; and $B$, possibly with some superscript, will denote a BM on a manifold generated by a polygonal domain or on $\R^{2}$. We note that the following lemma was inspired by ideas in \cite{Bellot}.

\begin{lemma}\label{lem:probabilistic_comparisons}

Let $D$ be a polygonal domain and $\delta > 0$ fixed.
\begin{enumerate}[label=(\roman*)]
    \item Let $x \in D$ with $d(x,\partial D) \geq \delta$, then we have that for any Borel sets $A_{1}\subset D^{*}$ and $A_{2}\subset \mathbb{R}^{2}$ with $B_{\delta}(x)\cap A_{1} = B_{\delta}(x)\cap A_{2}$, \begin{equation}
        \left|\int_{A_{1}}dy \; \eta_{D}(t;x,y)-\int_{A_{2}} dy\; p_{\mathbb{R}^{2}}(t;x,y)\right| \leq 4e^{-\delta^{2}/8t}.
    \end{equation}
    \item Let $\widetilde{e}$ be an edge of $\partial D$ and $x \in D$ with $d(x,\widetilde{e}) \leq \delta$ and $d(x,\partial D\backslash \widetilde{e}) \geq \delta$. Let $\mathbb{H}_{\widetilde{e}}$ denote the half-plane with $\widetilde{e}\subset \partial \mathbb{H}_{\widetilde{e}}$ and $x \in \mathbb{H}_{\widetilde{e}}$. Then we have that for any Borel sets $A_{3}\subset D^{*}$ and $A_{4}\subset \mathbb{H}_{\widetilde{e}}$ with $A_{3}\cap D^{*} \cap B_{\delta}(x) = A_{4}\cap \overline{\mathbb{H}_{\widetilde{e}}} \cap B_{\delta}(x)$, \begin{equation}
        \left|\int_{A_{3}}dy \; \eta_{D}(t;x,y)-\int_{A_{4}} dy\; \eta_{\mathbb{H}_{\widetilde{e}}}(t;x,y)\right| \leq 4e^{-\delta^{2}/8t}.
    \end{equation}
    \item Let $v$ be a vertex of $\partial D$ with interior angle $\gamma$. Let $W_{\gamma}$ be the infinite wedge of angle $\gamma$ with vertex at $v$ and suppose that $B_{2\delta}(v)\cap D= B_{2\delta}(v) \cap W_{\gamma}$. Then for any Borel sets $A_{5} \subset D^{*}$ and $A_{6} \subset W_{\gamma}^{*}$ with $A_{5} \cap D^{*}\cap B_{2\delta}(x) = A_{6} \cap W_{\gamma}^{*} \cap B_{2\delta}(x)$, then \begin{equation}
        \left|\int_{A_{5}}dy \; \eta_{D}(t;x,y)-\int_{A_{6}} dy\; \eta_{W_{\gamma}}(t;x,y)\right| \leq 4e^{-\delta^{2}/8t}.
    \end{equation}
\end{enumerate}
\end{lemma}

\begin{proof}

We do the comparisons at the manifold level. Let $N_{1} \subset M_{1}$ and $N_{2} \subset M_{2}$ be submanifolds of some Riemannian manifolds $M_{1}$ and $M_{2}$. Suppose there is an isometry $h: N_{1} \to N_{2}$, then we have that \begin{equation}
    \mathbb{P}_{x}^{M_{1}}(B_{t}^{(1)} \in A, \tau_{N_{1}} > t) = \mathbb{P}_{h(x)}^{M_{2}}(B_{t}^{(2)} \in h(A), \tau_{N_{2}} > t),
\end{equation} where: $A\subset M_{1}$ is a Borel set, $B^{(1)}$ is a BM on $M_{1}$ started at $x\in M_{1}$; $B^{(2)}$ is a BM on $M_{2}$ started at $h(x)\in M_{2}$; and, $\tau_{N_{1}}$ and $\tau_{N_{2}}$ are the first exit times of $B^{(1)}$ and $B^{(2)}$ from $N_{1}$ and $N_{2}$ respectively.

Throughout $M_{1}$ will always be the manifold generated by $D$ and $M_{2}$ will either be $\R^{2}$ or another manifold generated by a polygonal domain. For notational purposes, for a polygonal domain $D$ we define the set $$F(y,\delta,D):= \lbrace x\in D^{*}: |x-y| < \delta \rbrace$$ and for an RBM $X$ on $D$ we define $\tau_{X}(y,\delta)$ as the first exit time of $X$ from $F(y,\delta,D)$, that is $$\tau_{X}(y,\delta) := \inf\lbrace t\geq 0 : |X_{t} - y| \geq \delta\rbrace.$$ If $y=x$, then we simply denote this quantity by $\tau_{X}(\delta).$

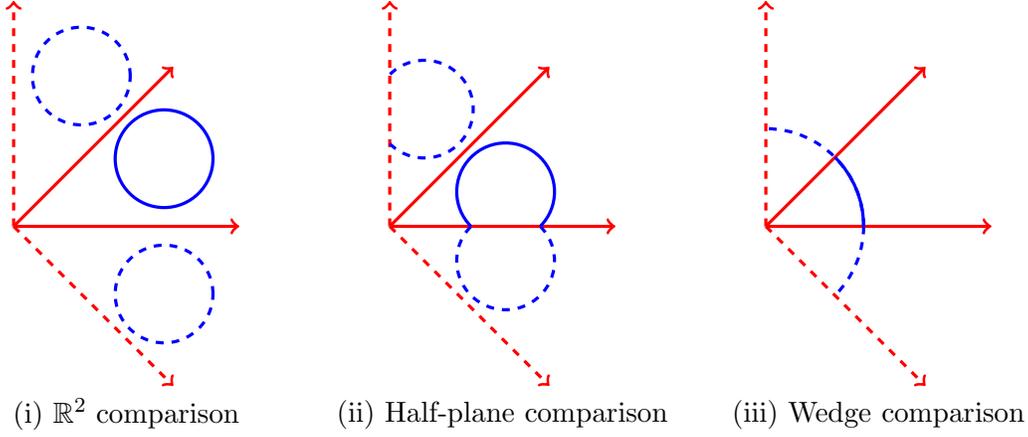
\begin{figure}
    \centering
    \begin{tikzpicture}
    \draw[red, very thick, ->] (0,0) -- (3,0);
    \draw[red, very thick, ->] (0,0) -- (2.121,2.121);
    \draw[red, very thick, ->,dashed] (0,0) -- (2.121,-2.121);
    \draw[red, very thick, ->,dashed] (0,0) -- (0,3);
    \draw[blue,very thick] (2,0.9) circle (0.65cm);
    \draw[blue,very thick,dashed] (0.9,2) circle (0.65cm);
    \draw[blue,very thick,dashed] (2,-0.9) circle (0.65cm);

    \draw[red, very thick, ->,xshift=5cm] (0,0) -- (3,0);
    \draw[red, very thick, ->,xshift=5cm] (0,0) -- (2.121,2.121);
    \draw[red, very thick, ->,xshift=5cm,dashed] (0,0) -- (0,3);
    \draw[red, dashed, very thick, ->,xshift=5cm] (0,0) -- (2.121,-2.121);
    \draw[blue,xshift= 5cm,very thick] (2,0) arc (-45:225:0.65cm);
    \draw[blue,xshift= 5cm,very thick,dashed] (2,0) arc (45:-225:0.65cm);
    \draw[blue,xshift= 5cm,very thick,dashed] (0,1.10) arc (-135:135:0.65cm);

    \draw[red, very thick, ->,xshift=10cm] (0,0) -- (3,0);
    \draw[red, very thick, ->,xshift=10cm] (0,0) -- (2.121,2.121);
    \draw[red, very thick, ->,xshift=10cm,dashed] (0,0) -- (2.121,-2.121);
    \draw[red, very thick, ->,xshift=10cm,dashed] (0,0) -- (0,3);
    \draw[blue,very thick,xshift= 10cm,dashed] (0+1.3,0) arc (0:90:1.3cm);
    \draw[blue,very thick,xshift= 10cm,dashed] (0+1.3,0) arc (0:-45:1.3cm);
    \draw[blue,very thick,xshift= 10cm] (0+1.3,0) arc (0:45:1.3cm);

    \node at (1.5,-2.5) {(i) $\mathbb{R}^{2}$ comparison};
    \node[xshift = 5cm] at (1.5,-2.5) {(ii) Half-plane comparison};
    \node[xshift = 10cm] at (1.5,-2.5) {(iii) Wedge comparison};
    \end{tikzpicture}
    \caption{Local illustrations of the preimages of subsets of a polygonal domain $D$, whose exit times we are interested in, on the manifold generated by $D$.}
    \label{fig:killing_times}
\end{figure}

For (i), we have $M_{2} = \mathbb{R}^{2}$ and let $B^{(2)}$ be a Brownian motion on $M_{2}$. Let $X := \Psi_{M_{1}} \circ B^{(1)}$ be an RBM on $D$ given as the projection of a BM $B^{(1)}$ on $M_{1}$. Via the construction of $X$ we immediately have that \begin{equation}
    \mathbb{P}_{x}^{D}(X_{t} \in A_{1}, \tau_{X}(\delta) > t) = \mathbb{P}_{x}^{M_{1}}(B_{t}^{(1)} \in \Psi_{M_1}^{-1}(A_{1}), \tau_{N_{1}} > t)
\end{equation} where $N_{1}$ is the connected component of $\Psi^{-1}_{M_{1}}(F(x,\delta,D))$ containing $x$. Now $N_{1}$ is isometric to the ball $N_{2}=B_{\delta}(x) \subset \R^{2}$, see Figure \ref{fig:killing_times}(i), and so we deduce that \begin{equation}
    \mathbb{P}_{x}^{M_{1}}(B_{t}^{(1)} \in \Psi_{M_{1}}^{-1}(A_{1}), \tau_{N_{1}} > t) = \mathbb{P}_{x}^{\R^{2}}(B_{t}^{(2)} \in  A_{2}, \tau_{N_{2}} > t) = \mathbb{P}_{x}^{\R^{2}}(B_{t}^{(2)} \in  A_{2}, \tau_{B^{(2)}}(\delta) > t),
\end{equation} by the restriction of the exit time and that $A_{1}\cap B_{\delta}(x) = A_{2} \cap B_{\delta}(x)$. Hence \begin{equation}
    \mathbb{P}_{x}^{D}(X_{t} \in A_{1}, \tau_{X}(\delta) > t) = \mathbb{P}_{x}^{\R^{2}}(B_{t}^{(2)} \in  A_{2}, \tau_{B^{(2)}}(\delta) > t),
\end{equation} and since $A_{1},A_{2}$ are arbitrary in the statement we further deduce that \begin{equation}
    \mathbb{P}_{x}^{D}(\tau_{X}(\delta) \leq t) = \mathbb{P}_{x}^{\R^{2}}(\tau_{B^{(2)}}(\delta) \leq t).
\end{equation} From \cite[\S 3]{vdBS90}, we know the bound $\mathbb{P}_{x}^{\R^{2}}(\tau_{B^{(2)}}(\delta) \leq t) \leq 4 e^{-\delta^{2}/8t}$. Now we have \begin{equation}\label{eqn:switching_trick}
    \begin{split}
        \int_{A_{1}}dy \;\eta_{D}(t;x,y) &= \mathbb{P}_{x}^{D}(X_{t} \in A_{1}) \\
        & = \mathbb{P}_{x}^{D}(X_{t} \in A_{1}, \tau_{X}(\delta) > t)+\mathbb{P}_{x}^{D}(X_{t} \in A_{1}, \tau_{X}(\delta) \leq t) \\
        & \leq \mathbb{P}_{x}^{\R^{2}}(B_{t}^{(2)}\in A_{2}, \tau_{B^{(2)}}(\delta) > t)+\mathbb{P}_{x}^{D}(\tau_{X}(\delta) \leq t) \\
        & \leq \mathbb{P}_{x}^{\R^{2}}(B_{t}^{(2)} \in A_{2}) + 4 e^{-\delta^{2}/8t} \\
        & = \left(\int_{A_{2}} dy \; p_{\R^{2}}(t;x,y) \right) + 4 e^{-\delta^{2}/8t}.
    \end{split}
\end{equation} Reversing the roles of $X$ and $B^{(2)}$, we obtain the result.

For (ii) the strategy is essentially the same as for (i) except we are comparing $X$ with an RBM $Y$ on the half plane $\mathbb{H}_{\widetilde{e}}$. The key subtlety is that the manifold generated by $\mathbb{H}_{\widetilde{e}}$ is simply $M_{2}=\mathbb{R}^{2}$ again and $B^{(2)}$ will again denote a Brownian motion in the plane. Except we see that $\Psi_{M_{1}}^{-1}(F(x,\delta,D))$ is now comprised of the disjoint union of sets which are each the union of two overlapping balls of radius $\delta$, see Figure \ref{fig:killing_times}(ii), and we take $N_{1}$ to be the connected component of $\Psi_{M_{1}}^{-1}(F(\delta,x,D))$ containing $x$. Arguing in the same way as in (i), \begin{equation}
    \mathbb{P}_{x}^{D}(X_{t} \in A_{3}, \tau_{X}(\delta) > t) = \mathbb{P}_{x}^{\mathbb{H}_{\widetilde{e}}}(Y_{t} \in A_{4}, \tau_{Y}(\delta) > t)
\end{equation} and hence we see that $\mathbb{P}_{x}^{D}(\tau_{X}(\delta) \leq t) = \mathbb{P}_{x}^{\mathbb{H}_{\widetilde{e}}}( \tau_{Y}(\delta) \leq t)$. It is easy to see that \begin{equation}
    \mathbb{P}_{x}^{\mathbb{H}_{\widetilde{e}}}( \tau_{Y}(\delta) \leq t) \leq \mathbb{P}_{x}^{\R^{2}}(\tau_{B^{(2)}}(\delta) \leq t) \leq 4 e^{-\delta^{2}/8t}.
\end{equation} Hence following the argument as in \eqref{eqn:switching_trick} replacing $B^{(2)}$ with $Y$, we obtain the result.

For (iii) again the strategy is essentially the same as for (i) but now $M_{2}$ is the manifold generated by the wedge $W_{\gamma}$, $N_{1}$ and $N_{2}$ are the connected components of $\Psi_{M_{1}}^{-1}(F(v, 2\delta, D))$ and $\Psi_{M_{2}}^{-1}(F(v, 2\delta, W_{\gamma}))$ containing $x$ respectively, see Figure \ref{fig:killing_times}(iii), and we are comparing $X$ with an RBM $Y$ on the infinite wedge $W_\gamma$. But $N_{1}$ and $N_{2}$ are clearly isometric and thus one determines immediately as above that
\begin{equation}
   \mathbb{P}_{x}^{D}(X_{t} \in A_{5}, \tau_{X}(v,2\delta) > t) = \mathbb{P}_{x}^{W_{\gamma}}(Y_{t} \in A_{6}, \tau_{Y}(v,2\delta) > t)
\end{equation} and hence $\mathbb{P}_{x}^{D}(\tau_{X}(v,2\delta) \leq t) = \mathbb{P}_{x}^{W_{\gamma}}(\tau_{Y}(v,2\delta) \leq t)$.
The radial component of RBM in the Neumann wedge $W_{\gamma}^{*}$ is the same as that of a BM in $\mathbb{R}^{2}$.
So, since $x\in F(v,\delta,W_{\gamma})$, we see that $\mathbb{P}_{x}^{W_{\gamma}}(\tau_{Y}(v,2\delta) \leq t) \leq 4e^{-\delta^{2}/8t}$. Hence following the argument as in \eqref{eqn:switching_trick} replacing $B^{(2)}$ with $Y$, one  obtains the result.
\end{proof}

The proof of Theorem \ref{thm:model_justifications} now follows by combining Lemma \ref{lem:probabilistic_comparisons} with the model computations proven in the previous section. Let $R$ and $\delta$ be as they are defined in Section 2.

For $\widetilde{D}(R,\delta)$, we know that each $x\in \widetilde{D}(R,\delta)$ satisfies $d(x,\partial D) \geq \delta$ and so we immediately deduce that from Lemma \ref{lem:probabilistic_comparisons}(i) taking $A_{1}=A_{2}=\widetilde{D}$ that \begin{equation}
\begin{split}
    & \left| \int_{\widetilde{D}(R,\delta)} dx \int_{\widetilde{D}} dy \; \eta_{D}(t;x,y) - \int_{\widetilde{D}(R,\delta)} dx \int_{\widetilde{D}} dy \; p_{\R^{2}}(t;x,y) \right| \\
    & \qquad \qquad \leq \int_{\widetilde{D}(R,\delta)} dx \left| \int_{\widetilde{D}} dy \; \eta_{D}(t;x,y) -  \int_{\widetilde{D}} dy \; p_{\R^{2}}(t;x,y)\right| \\
    & \qquad \qquad \leq 4 |\widetilde{D}(R,\delta)| e^{-\delta^{2}/8t}.
\end{split}
\end{equation} This inequality also holds if we replace $\widetilde{D}(R,\delta)$ with an open cusp or with a rectangle $T_{\widetilde{e}}(R,\delta)$ lying on an open edge. For a sector $\widetilde{S}_{v}(R)$ at an open vertex, we note that $B_{2R}(v) \subset D$ and so for each $x\in \widetilde{S}_{v}(R)$ we have that $d(x,\partial D) \geq R$ and we can apply Lemma \ref{lem:probabilistic_comparisons}(i) as in the other cases.

For Neumann cusps and rectangles lying on Neumann edges, we apply Lemma \ref{lem:probabilistic_comparisons}(ii) in the same way except that we take $A_{3} = \widetilde{D}$ and $A_{4} = \mathbb{H}_{\widetilde{e}}$. Moreover, for sectors $\widetilde{S}_{v}(R)$ at NN, NON, and NOON vertices with interior angle $\sigma$ with respect to $D$, we know that for all $x\in \widetilde{S}_{v}(R)$ we have that $d(x,v) \leq R$ and that $S_{2R}(v) \cap D = S_{2R}(v) \cap W_{\sigma}$, for the infinite wedge $W_{\sigma}$ by our definition of $R$. Thus, we can apply Lemma \ref{lem:probabilistic_comparisons}(iii) as above with $A_{5} = \widetilde{D}$ and $A_{6}$ a suitable infinite wedge based at $v$ as described in Section 3 where we calculated the model computations for sectors at Neumann vertices.

\appendix

\section{Model calculations for other vertices}

The calculations in Section 3 can be extended to vertices with an arbitrary number of incident edges in $\widetilde{E}$. The case of open vertices lying in $D$ with an arbitrary number of incident edges which are all open is dealt with in \cite{vdBG16}. Here we outline the case when the vertex lies on the Neumann boundary $\partial D$. The notation used below is the same as that used in the proofs of Theorems \ref{NON_wedge} and \ref{NOON_wedge} in Section 3

We choose such a vertex $v\in \widetilde{V}$ with interior angle with respect to $D$ denoted $\sigma \in (0,2\pi)$. Translating $v$ to the origin and rotating as necessary, choose $R>0$ sufficiently small so that \begin{equation}
B_{R}(0) \cap \widetilde{D} = B_{R}(0) \cap \left(\bigcup_{i=1}^{k}W_{\lambda_{i}}^{\rho_{i}}\right)
\end{equation} where \begin{equation}
W_{\lambda_{i}}^{\rho_{i}} := \lbrace (r,\phi) : r > 0, \lambda_{i} < \phi < \rho_{i} \rbrace
\end{equation} and $0\leq \lambda_{1} < \rho_{1} < \lambda_{2} < \rho_{2} < \cdots < \lambda_{k} < \rho_{k} \leq \sigma$. Then, as in Section 2, we set $\widetilde{S}_{v}(R) = B_{R}(0) \cap \widetilde{D}$ and we model the heat content contribution from $\widetilde{S}_{v}(R)$ by \begin{equation}
    \begin{split}
        & \sum_{i=1}^{k}\sum_{j=1}^{k}\mathcal{L}^{-1} \left\lbrace \int_{0}^{R} r dr \int_{\lambda_{i}}^{\rho_{i}} d\phi \int_{0}^{\infty} r_{0} dr_{0} \int_{\lambda_{j}}^{\rho_{j}} d\phi_{0} \; G_{W_{\gamma}}(s;r,\phi,r_{0},\phi_{0}) \right\rbrace(t)  \\
        & \qquad = \sum_{i=1}^{k}\mathcal{L}^{-1} \left\lbrace \int_{0}^{R} r dr \int_{\lambda_{i}}^{\rho_{i}} d\phi \int_{0}^{\infty} r_{0} dr_{0} \int_{\lambda_{i}}^{\rho_{i}} d\phi_{0} \; G_{W_{\gamma}}(s;r,\phi,r_{0},\phi_{0}) \right\rbrace(t) \\
        & \qquad \quad + \sum_{i=1}^{k}\sum_{\substack{1\leq j \leq k \\ j\neq i}}\mathcal{L}^{-1} \left\lbrace \int_{0}^{R} r dr \int_{\lambda_{i}}^{\rho_{i}} d\phi \int_{0}^{\infty} r_{0} dr_{0} \int_{\lambda_{j}}^{\rho_{j}} d\phi_{0} \; G_{W_{\gamma}}(s;r,\phi,r_{0},\phi_{0}) \right\rbrace(t).
    \end{split}
\end{equation} For $0\leq \lambda < \rho \leq \sigma$, we define \begin{equation}
\hat{d}(\sigma,\rho,\lambda) := \begin{cases}
0, & \lambda = 0, \rho = \sigma \\
\hat{b}(\sigma,\rho) & \lambda = 0,\rho < \sigma \\
\hat{b}(\sigma,\sigma - \lambda) & \lambda > 0,\rho = \sigma \\
\hat{c}(\sigma,\rho,\lambda) & \lambda > 0,\rho < \sigma \\
\end{cases},
\end{equation} where $\hat{b}$ and $\hat{c}$ are defined in the statements of Theorems \ref{NON_wedge} and \ref{NOON_wedge} respectively. Then, from Lemma \ref{lem:NN} and Theorems \ref{NON_wedge} and \ref{NOON_wedge}, we know that \begin{equation}\label{eqn:diag_terms}
\begin{split}
& \int_{0}^{R} r dr \int_{\lambda}^{\rho}d\phi \int_{0}^{\infty} r_{0} dr_{0} \int_{\lambda}^{\rho}d\phi_{0} \; \eta_{W_{\sigma}}(t;r,\phi,r_{0},\phi_{0}) \\
& \qquad \qquad = \frac{1}{2}(\rho-\lambda)R^{2} - \frac{(2-\mathds{1}_{\lbrace 0\rbrace}(\lambda)-\mathds{1}_{\lbrace \sigma\rbrace}(\rho))R}{\pi^{1/2}}t^{1/2} + \hat{d}(\sigma,\rho,\lambda)t \\
& \qquad \qquad \quad + (2-\mathds{1}_{\lbrace 0\rbrace}(\lambda)-\mathds{1}_{\lbrace \sigma\rbrace}(\rho)) \frac{R}{\pi^{1/2}}t^{1/2}\int_{1}^{\infty}\frac{dv}{v^{2}}\int_{0}^{1}dy\frac{y}{(1-y^{2})^{1/2}} e^{-R^{2}y^{2}v^{2}/4t} + O(e^{-C/t}),
\end{split}
\end{equation} from some constant $C>0$ depending on $R$, $\sigma$, $\rho$, and $\lambda$ (note that when $\lambda=0$ and $\rho=\sigma$ we consider a Neumann vertex so there is no remainder as in Lemma \ref{lem:NN}).

Suppose W.L.O.G. that $0\leq \lambda < \rho < \lambda' < \rho'\leq \sigma$. By Fubini's theorem we observe that \begin{equation}
\begin{split}
    \int_{\lambda}^{\rho} d\phi \int_{\lambda'}^{\rho'}d\phi_{0} \; \Phi_{\sigma}(\theta,\phi,\phi_{0}) & = \int_{0}^{\rho}d\phi\int_{0}^{\rho'}d\phi_{0} \; \Phi_{\sigma}(\theta,\phi,\phi_{0}) + \int_{0}^{\lambda}d\phi \int_{0}^{\lambda'}d\phi_{0} \; \Phi_{\sigma}(\theta,\phi,\phi_{0}) \\
    & \quad - \int_{0}^{\lambda}d\phi\int_{0}^{\rho'}d\phi_{0} \; \Phi_{\sigma}(\theta,\phi,\phi_{0}) - \int_{0}^{\rho}d\phi \int_{0}^{\lambda'}d\phi_{0} \; \Phi_{\sigma}(\theta,\phi,\phi_{0}).
\end{split}
\end{equation} Using \eqref{eqn:NOON_Sum} on each of these four double integrals, we see that \begin{equation}
\begin{split}
    \int_{\lambda}^{\rho} d\phi \int_{\lambda'}^{\rho'}d\phi_{0} \; \Phi_{\sigma}(\theta,\phi,\phi_{0}) & = \frac{\sinh(\pi\theta)}{\theta^{2}\sinh(\sigma \theta)} \big( \cosh((\sigma-\rho'-\rho)\theta) - \cosh((\sigma-\rho'+\rho)\theta) \\
    & \; \qquad \qquad \qquad + \cosh((\sigma-\lambda'-\lambda)\theta) - \cosh((\sigma-\lambda'+\lambda)\theta) \\
    & \; \qquad \qquad \qquad - \cosh((\sigma-\rho'-\lambda)\theta) + \cosh((\sigma-\rho'+\lambda)\theta) \\
    & \; \qquad \qquad \qquad - \cosh((\sigma-\lambda'-\rho)\theta) + \cosh((\sigma-\lambda'+\rho)\theta) \big) \\
    & =: \frac{\sinh(\pi\theta)}{\theta^{2}\sinh(\sigma \theta)} \hat{g}(\theta;\sigma,\rho,\lambda,\rho',\lambda').
    \end{split}
\end{equation} Using similar arguments as in the proof of Theorem \ref{NOON_wedge} and repeated use of Lemma \ref{angle_remainder}, one can deduce that

\begin{equation}\label{eqn:corrective_terms}
\begin{split}
    & \mathcal{L}^{-1} \left\lbrace \int_{0}^{R} r dr \int_{\lambda}^{\rho} d\phi \int_{0}^{\infty} r_{0} dr_{0} \int_{\lambda'}^{\rho'} d\phi_{0} \; G_{W_{\gamma}}(s;r,\phi,r_{0},\phi_{0}) \right\rbrace(t) \\
    & \qquad \qquad \qquad = \int_{0}^{\infty}d\theta \; \frac{\cosh(\frac{\pi}{2}\theta)\hat{g}(\theta;\sigma,\rho,\lambda,\rho',\lambda')}{2\sinh(\sigma\theta)\sinh(\frac{\pi}{2}\theta)} + O(e^{-C'/t}) \\
    & \qquad \qquad \qquad =: \hat{h}(\sigma,\rho,\lambda,\rho',\lambda')+ O(e^{-C'/t}),
\end{split}
\end{equation} where $C' >0$ is a constant depending on $R$ and the angles $\sigma,\lambda,\rho,\lambda',\rho'$. Using \eqref{eqn:diag_terms} and \eqref{eqn:corrective_terms}, we see that the model heat content contribution of $\widetilde{S}_{v}(R)$ is \begin{equation}
\begin{split}
    |\widetilde{S}_{v}(R)| - \frac{(2k-\mathds{1}_{\lbrace 0 \rbrace}(\lambda_{1})-\mathds{1}_{\lbrace \sigma \rbrace}(\rho_{k}))R}{\pi^{1/2}}t^{1/2} + \left( \sum_{i=1}^{k}\hat{d}(\sigma,\rho_{i},\lambda_{i})+ \sum_{i=1}^{k}\sum_{\substack{1\leq j \leq k \\ j\neq i}} \hat{h}(\sigma, \rho_{i},\lambda_{i},\rho_{j},\lambda_{j})\right)t \\
    + (2k-\mathds{1}_{\lbrace 0\rbrace}(\lambda_{1})-\mathds{1}_{\lbrace \sigma\rbrace}(\rho_{k})) \frac{R}{\pi^{1/2}}t^{1/2}\int_{1}^{\infty}\frac{dv}{v^{2}}\int_{0}^{1}dy\frac{y}{(1-y^{2})^{1/2}} e^{-R^{2}y^{2}v^{2}/4t} + O(e^{-C''/t})
    \end{split}
\end{equation} for some constant $C''>0$ depending on $R$ and the angles $\sigma,\lambda_{1},\rho_{1},\ldots, \lambda_{k},\rho_{k}$.

To extend Theorem \ref{thm:main} to include vertices with an arbitrary number of incident edges in $\widetilde{E}$, one would need to construct a partition analogously to the construction given in Section 2. We remark that we believe these ideas can be used for such wedges where instead of Neumann boundary conditions we have Dirichlet boundary conditions or mixed Dirichlet-Neumann boundary conditions. This would give rise to analogues of Theorem \ref{thm:main} in the Dirichlet-Open-Open-Dirichlet (DOOD) case, the Dirichlet-Open-Open-Neumann (DOON) case, and so on.

\section{Vertices with interior angle $2\pi$ and unbounded domains}

Theorem \ref{thm:main} readily extends to include vertices of $D$ with interior angle $2\pi$ as we shall now describe. To make sense of this we first need to consider a generalised Neumann boundary condition in the spirit of \cite{C92}. Consider the manifold metric \begin{equation}\label{eqn:manifold_metric}
    \widehat{d}(x,y) := \inf\lbrace \ell(\gamma) : \text{$\gamma$ is piecewise $C^{1}$ from $x$ to $y$}\rbrace
\end{equation} on $D$ where $\ell(\gamma)$ denotes the length of $\gamma$. The topology generated on $D$ by this metric is precisely the Euclidean topology however the completion $\widehat{D}$ of $D$ in this metric is not a subset of $\mathbb{R}^{2}$, rather an abstract space. We have a new boundary $\partial \widehat{D} := \widehat{D}\setminus D$ for which the normal derivative $\widehat{n}$ on the boundary can be defined almost everywhere in the obvious way. The heat equation with Neumann boundary conditions then becomes \begin{equation}
    \begin{cases}
    \displaystyle \frac{\partial u}{\partial t}(t;x) = \Delta u(t;x), & t>0, \; x\in D, \vspace{1em}\\
    \displaystyle \frac{\partial u}{\partial \widehat{n}}(t;x) = 0, & t>0, \; x\in \partial \widehat{D}\text{ a.e.}\, .
    \end{cases}
\end{equation} Existence of a solution to this problem comes from the existence of a Neumann heat kernel using a simple adaptation of the manifold construction by reflection discussed in Section 4 following \cite[\S 3]{GM72}. Uniqueness of the solution to the heat equation can be obtained via the classical energy method using a generalised Green's formula, see \cite[Thm 4.5]{C93}, in light of $\widehat{D}$ being a pseudo Jordan domain, see \cite{C92}. Alternatively, one can use a suitable generalisation of the maximum principle to prove this.

For a vertex with interior angle $2\pi$ the comparison to make is with the $2\pi$ wedge $W_{2\pi}:= \lbrace (r,\phi): r > 0, 0 < \phi < 2\pi\rbrace$. We impose the generalised Neumann boundary condition on $\partial \widehat{W}_{2\pi}$. Existence of a Neumann heat kernel in this case again comes from an adaptation of the methods in \cite[\S 3]{GM72}, or indeed by observation comparing with the Neumann heat kernel for the wedge with angle $<2\pi$, i.e. the inverse Laplace transform of \begin{equation}
    G_{W_{2\pi}}(s;r,\phi,r_{0},\phi_{0}) = \frac{1}{\pi^{2}} \int_{0}^{\infty} d\theta \; K_{i\theta}(r\sqrt{s}) K_{i\theta}(r_{0}\sqrt{s}) \Phi_{2\pi}(\theta,\phi,\phi_{0}),
\end{equation} using the notation from Section 3. For model computations with this Green's function, the results of Section 3 and Appendix A readily extend to this case and the comparisons work in the exact same way as in Section 4. The only point left to prove is the uniqueness of the Neumann heat kernel for the $2\pi$ wedge. We do not do this directly here and leave an outline of the idea in Remark \ref{rem:2pi_wedge}.

Theorem \ref{thm:main} also extends to the case of polygonal domains whose boundary is a collection of Neumann edges and open edges.
Consider a polygonal domain $D$ with each interior angle strictly less than $2\pi$. Now pick a sub-collection $\partial D_{+} \subsetneq \partial D$ of edges of $D$ and define a new domain $\Omega := \mathbb{R}^{2}\setminus \partial D_{+}$. In what follows, $\Omega$ will be a domain of this form and $\widehat{\Omega}$ will denote its closure with respect to the manifold metric. On $\Omega$ we can consider the (generalised) Neumann heat equation with initial datum $f\in L^{\infty}(\Omega)$, i.e.

\begin{equation}\label{eqn:gen_Neu_heat_eqn}
    \begin{cases}
    \displaystyle \frac{\partial u}{\partial t}(t;x) = \Delta u(t;x), & t>0, \, x\in \Omega, \vspace{1em}\\
    \displaystyle \frac{\partial u}{\partial \widehat{n}}(t;x) = 0, & t>0, \, x\in \partial \widehat{\Omega}\text{ a.e.}, \vspace{1em}\\
    \lim_{t\downarrow 0}u(t;x) = f(x), & x\in \Omega.
    \end{cases}
\end{equation}

Via a simple adaptation of the method in \cite{GM72}, one can construct a suitable generalised Neumann heat kernel on $\Omega$ from the fundamental solution to the heat equation on a manifold constructed from reflected copies of $\Omega$, and so there exists a bounded solution to \eqref{eqn:gen_Neu_heat_eqn}. In fact, such a solution is unique, see Proposition \ref{prop:generalised_neumann_uniq}. In particular, we can take $f=\mathds{1}_{\widetilde{D}}$ for some polygonal subdomain $\widetilde{D}\subset D \subset \Omega$ and consider the heat content of $\widetilde{D}$ and obtain a small-time asymptotic formula of the form in Theorem \ref{thm:main} in this case. The comparisons work on the manifold level as in Section 4, which justifies the local approximations of the heat content in small-time.

Finally, we prove the uniqueness of a bounded solution to \eqref{eqn:gen_Neu_heat_eqn}.

\begin{prop}\label{prop:generalised_neumann_uniq} Let $\Omega = \mathbb{R}^{2}\setminus \partial D_{+}$ as above. For $f\in L^{\infty}(\Omega)$, there exists a unique bounded solution to \eqref{eqn:gen_Neu_heat_eqn}.
\end{prop}

To prove Proposition \ref{prop:generalised_neumann_uniq}, we employ a result from \cite{C93} (in particular, page 307). This result relies on three properties. We address the first one in Lemma \ref{lem:regular_space}.

\begin{lemma}\label{lem:regular_space}
We have that $C_{0}(\widehat{\Omega})\cap H^{1}(\Omega)$ is dense in $H^{1}(\Omega)$ with respect to the $H^{1}$-norm and is dense in $C_{0}(\widehat{\Omega})$ with respect to the uniform norm.
\end{lemma}

\begin{figure}
    \centering
    \begin{tikzpicture}
    \draw[red, very thick] (0,0) -- (2,0) -- (4,2);
    \draw[dashed,thick,blue] (0,0) circle (1.5);
    \draw[dashed,thick,blue] (4,2) circle (1.5);
    \draw[dashed,rounded corners=5mm,thick,orange] (0.5,0) -- (0.5,1.5) -- (2,1.5) -- (2.5,2.5) -- (3.5,1.5);
    \draw[dashed,rounded corners=5mm,thick,orange] (0.75,0) -- (0.75,-1.5) -- (2,-1.5) -- (4.75,0.75) -- (3.75,1.75);
    \draw[dashed,rounded corners=5mm,thick] (-0.75,0) -- (-0.75,1) -- (2,1) -- (4,3) -- (5,2) -- (2,-1) -- (-0.75,-1) -- cycle;
    \node at (6,0) {$U_{1}$};
    \node at (0,-0.5) {\color{blue} $U_{2}$};
    \node at (4,3.2) {\color{blue} $U_{3}$};
    \node at (2,0.5) {\color{orange} $U_{4}$};
    \node at (2,-0.5) {\color{orange} $U_{5}$};
    \node at (4,2.25) {\color{red} $\partial D_{+}$};
    \end{tikzpicture}
    \caption{An illustration an open cover $U_{1},\ldots, U_{5}$ of the case when $\partial D_{+}$ is comprised of two edges. $U_{1}$ is as described in the proof of Lemma \ref{lem:regular_space}, $U_{2}$ and $U_{3}$ are vertex sets in this cover and $U_{4}$ and $U_{5}$ are boundary sets.}
    \label{fig:regular_covering}
\end{figure}
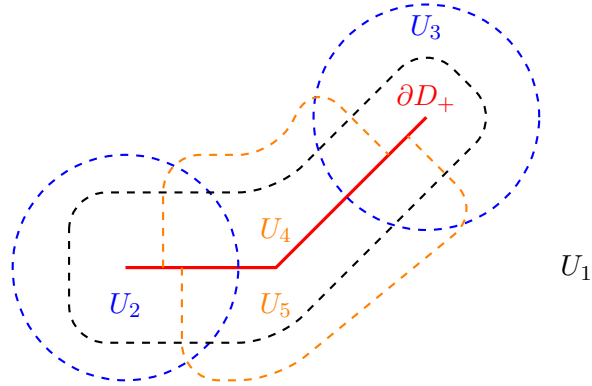

\begin{proof}
Take a finite cover $\lbrace U_{i}\rbrace$ of $\widehat{\Omega}$ by open sets in the following way. Let $V_{+}$ denote the collection of endpoints of connected components of $\partial D_{+}$, note that for each $v\in V_{+}$, there is a unique corresponding $\overline{v} \in V_{+}$ such that there is a path through $\partial D_{+}$ to $\overline{v}$. We call $\overline{v}$ the connected partner of $v$. Fix $r>0$ such that for all $v_{1},v_{2}\in V_{+}$, $B_{r}(v_{1}) \cap B_{r}(v_{2}) = \emptyset$. These become our first collection of open sets of $\widehat{\Omega}$, which we shall call vertex sets in our cover. Then along each side of the boundary between a given vertex $v\in V_{+}$ and its connected partner $\overline{v}$, pick two disjoint bounded open sets $U_{1}$ and $U_{2}$ on each side of the boundary such that $B_{r}(v)\cap U_{j}, B_{r}(\overline{v})\cap U_{j} \neq \emptyset$, $\widehat{d}(V_{+},U_{j}) > 0$ and are at a positive distance from any other part of the boundary of $\partial \widehat{\Omega}$. These we call boundary sets and form our next part of the open cover. Finally, choose an open set $U$ that covers the rest of $\widehat{\Omega}$ and lies at a positive distance from $\partial \widehat{\Omega}$, we denote this set by $U_{1}$. Now we have our open cover $\lbrace U_{i}\rbrace$, we take $\lbrace \chi_{i} \rbrace$ to be a partition of unity subordinate to this open cover. See Figure \ref{fig:regular_covering} for a visualisation of such an open cover for an example case where $\partial D_{+}$ is comprised of two edges.

Pick a function $u\in H^{1}(\Omega)$. Then we can view $\chi_{1}u$ as lying in $H^{1}(\mathbb{R}^{2})$ and so we can find a sequence $u_{n}^{(1)} \in H^{1}(\mathbb{R}^{2})\cap C_{0}(\mathbb{R}^{2})$ such that $u_{n}^{(1)} \to \chi_{1}u$ in the $H^{1}$-norm as $n\to\infty$. In particular, we can choose the $u_{n}^{(1)}$ such that their supports are contained $U_{1}$. For a boundary set $U_{i}$, $i\neq 1$, we can view $U_{i}$ as a subset of a bounded polygonal domain $D'$ with interior angles $<2\pi$. Then we can view $\chi_{i}u$ as lying in $H^{1}(D')$ and so we can find a sequence of functions $u_{n}^{(i)} \in H^{1}(D')\cap C_{0}(D')$ converging to $\chi_{i}u$ in the $H^{1}$-norm and moreover we can choose the $u_{n}^{(i)}$ such that their supports are contained in $U_{i}$. Finally for $U_{i} = B_{r}(v)$ a vertex set in our cover, we can view $U_{i}$ as the ball of radius $r$ centred at the origin with the slit $[0,r) \times \lbrace 0\rbrace$ removed. Then we know from \cite{C92} that this is a pseudo Jordan domain and that there exist functions $u_{n}^{(i)} \in H^{1}(U_{i})\cap C_{0}(U_{i})$ converging to $\chi_{i}u$ in the $H^{1}$-norm and again we can choose the $u_{n}^{(i)}$ such that their supports are contained in $U_{i}$.

Since our open cover was finite, putting this all together we see that \begin{equation}
    H^{1}(\Omega)\cap C_{0}(\widehat{\Omega}) \ni \sum_{i} u_{n}^{(i)} \to \sum_{i} \chi_{i} u = u \in H^{1}(\Omega)
\end{equation} as $n\to \infty$. Since $u$ was arbitrary, we see that $H^{1}(\Omega)\cap C_{0}(\widehat{\Omega})$ is dense in $H^{1}(\Omega)$ in the $H^{1}$-norm. The fact that $H^{1}(\Omega)\cap C_{0}(\widehat{\Omega})$ is dense in $C_{0}(\widehat{\Omega})$ with respect to the uniform norm can be deduced in essentially the same way.
\end{proof}

The remaining two required properties of $\Omega$ follow by observation:
\begin{itemize}
    \item There exists a regular exhaustion of $(\Omega_{n})_{n \geq 1}$ of $\Omega$, that is a collection of smooth bounded connected domains with $\overline{\Omega}_{n} \subset \Omega_{n+1}$ and $\bigcup_{n\geq 1} \Omega_{n} = \Omega$, such that \begin{equation}
    \sup_{n} |(\partial \Omega_{n})\cap B_{r}(0)| < \infty, \enspace \forall r>0.
\end{equation}
\item For each $y\in \widehat{\Omega}$ and $i=1,2$, there exists an open neighbourhood $U$ of $y$ and $f\in H^{1}(\Omega)$ such that $f(x) = x_{i}$ for all $x\in U\cap \Omega$.
\end{itemize}

With these  two properties and the result from Lemma \ref{lem:regular_space} in hand, the following result from \cite{C93} holds.

\begin{lemma}\label{lem:skorokhod}
Let $\Omega = \mathbb{R}^{2}\setminus \partial D_{+}$ as above.
There exists a process $\widehat{X}$ on $\widehat{\Omega}$ such that under the obvious quasi-continuous inclusion map $\iota : \widehat{\Omega} \to \overline{\Omega}$, we see that $X := \iota \circ \widehat{X}$ has the following Skorokhod decomposition \begin{equation}
   X_{t} = X_{0}+B_{t} + \int_{0}^{t} \widehat{n}(\widehat{X}_{s})\, dL_{s}, \quad t \geq 0,
\end{equation} where $B$ is a 2-dimensional Brownian motion martingale additive functional of $\widehat{X}$ and $L$ is a positive continuous additive functional associated with the surface measure of $\widehat{\Omega}$.
\end{lemma}

The proof of Proposition \ref{prop:generalised_neumann_uniq} now follows by similar arguments to those in \cite{GM72}.

\begin{proof}[Proof of Proposition \ref{prop:generalised_neumann_uniq}]
Existence was already addressed in our preceding discussion. It remains to prove uniqueness. We may assume that $\widehat{X}$ does not hit any of the vertices $\widehat{V}$ of $\partial \widehat{\Omega}$. And so for any $u\in C^{\infty}((0,\infty) \times \widehat{\Omega}\backslash \widehat{V})$, we can deduce a suitable variation of It\^{o}'s lemma from Lemma \ref{lem:skorokhod} in this case. Namely, that \begin{equation}
\begin{split}
    u(t;\widehat{X}_{t}) - u(0;\widehat{X}_{0}) & = \int_{0}^{t} \nabla u(s;\widehat{X}_{s}) \cdot dX_{s} + \int_{0}^{t} \left(\frac{\partial }{\partial t} + \Delta \right)u(s;\widehat{X}_{s}) \, ds\\
    & = \int_{0}^{t} \nabla u(s;\widehat{X}_{s}) \cdot dB_{s} + \int_{0}^{t} \nabla u(s;\widehat{X}_{s}) \cdot \widehat{n}(\widehat{X}_{s}) \, dL_{s} + \int_{0}^{t} \left(\frac{\partial }{\partial t} + \Delta \right)u(s;\widehat{X}_{s})\, ds.
\end{split}
\end{equation} And  since $u$ is a solution to the heat equation on $\Omega$ with (generalised) Neumann boundary condition, we see that \begin{equation}
    u(T-t;\widehat{X}_{t}) - u(T;\widehat{X}_{0}) = \int_{0}^{t} \nabla u(T-s;\widehat{X}_{s}) \cdot dB_{s},
\end{equation} from which we can deduce the uniqueness of $u$ following the proof of Theorem 3.6 in \cite{GM72}.
\end{proof}

\begin{rem}\label{rem:2pi_wedge}
The uniqueness of bounded solutions to the heat equation on the $2\pi$ wedge with bounded initial datum can be deduced in a similar way by decomposing a reflecting Brownian motion on the $2\pi$ wedge into its radial and angular parts and following the lines of the final part of the proof of Proposition \ref{prop:generalised_neumann_uniq}.
\end{rem}

\textbf{Acknowledgements}: We would like to thank Prof. Michiel van den Berg, Prof. Kryzstof Burdzy, Dr Megan Griffin-Pickering, and Prof. David Williams for their helpful comments and suggestions on this work. Sam Farrington was supported by an EPSRC DTG Studentship.

\bibliographystyle{plain}
\bibliography{biblio}

\begin{thebibliography}{10}

\bibitem{BB93}
R.~F. Bass and K.~Burdzy.
\newblock On domain monotonicity of the {N}eumann heat kernel.
\newblock {\em J. Funct. Anal.}, 116(1):215--224, 1993.

\bibitem{Bellot}
B.~A. Bellot.
\newblock {\em On reflected {B}rownian motion in two dimensions}.
\newblock ProQuest LLC, Ann Arbor, MI, 1975.
\newblock Thesis (Ph.D.)--University of California, Berkeley.

\bibitem{vdB13}
M.~van~den Berg.
\newblock Heat flow and perimeter in {$\R^m$}.
\newblock {\em Potential Anal.}, 39(4):369--387, 2013.

\bibitem{vdBGG20}
M.~van~den Berg, P.~B. Gilkey, and K.~Gittins.
\newblock Heat flow from polygons.
\newblock {\em Potential Anal.}, 53(3):1043--1062, 2020.

\bibitem{vdBG16}
M.~van~den Berg and K.~Gittins.
\newblock On the heat content of a polygon.
\newblock {\em J. Geom. Anal.}, 26(3):2231--2264, 2016.

\bibitem{vdBS88}
M.~van~den Berg and S.~Srisatkunarajah.
\newblock Heat equation for a region in {${\bf R}^2$} with a polygonal
  boundary.
\newblock {\em J. London Math. Soc. (2)}, 37(1):119--127, 1988.

\bibitem{vdBS90}
M.~van~den Berg and S.~Srisatkunarajah.
\newblock Heat flow and {B}rownian motion for a region in {${\bf R}^2$} with a
  polygonal boundary.
\newblock {\em Probab. Theory Related Fields}, 86(1):41--52, 1990.

\bibitem{CZ94}
R.~A. Carmona and W.~I. Zheng.
\newblock Reflecting {B}rownian motions and comparison theorems for {N}eumann
  heat kernels.
\newblock {\em J. Funct. Anal.}, 123(1):109--128, 1994.

\bibitem{C86}
I.~Chavel.
\newblock Heat diffusion in insulated convex domains.
\newblock {\em J. London Math. Soc. (2)}, 34(3):473--478, 1986.

\bibitem{C92}
Z-Q Chen.
\newblock Pseudo {J}ordan domains and reflecting {B}rownian motions.
\newblock {\em Probab. Theory Related Fields}, 94(2):271--280, 1992.

\bibitem{C93}
Z-Q Chen.
\newblock On reflecting diffusion processes and {S}korokhod decompositions.
\newblock {\em Probab. Theory Related Fields}, 94(3):281--315, 1993.

\bibitem{GM72}
G.~Gallavotti and H.~P. McKean.
\newblock Boundary conditions for the heat equation in a several-dimensional
  region.
\newblock {\em Nagoya Math. J.}, 47:1--14, 1972.

\bibitem{GR15}
I.~S. Gradshteyn and I.~M. Ryzhik.
\newblock {\em Table of integrals, series, and products}.
\newblock Elsevier/Academic Press, Amsterdam, eighth edition, 2015.
\newblock Translated from the Russian, Translation edited and with a preface by
  Daniel Zwillinger and Victor Moll, Revised from the seventh edition
  [MR2360010].

\bibitem{H94}
E.~P. Hsu.
\newblock A domain monotonicity property of the {N}eumann heat kernel.
\newblock {\em Osaka J. Math.}, 31(1):215--223, 1994.

\bibitem{H84}
P.~Hsu.
\newblock {\em R{eflecting} {Brownian} {Motion}, {Boundary} {Local} {Time}
  {and} {the} {Neumann} {Problem}}.
\newblock ProQuest LLC, Ann Arbor, MI, 1984.
\newblock Thesis (Ph.D.)--Stanford University.

\bibitem{K66}
M.~Kac.
\newblock Can one hear the shape of a drum?
\newblock {\em Amer. Math. Monthly}, 73(4, part II):1--23, 1966.

\bibitem{K89}
W.~S. Kendall.
\newblock Coupled {B}rownian motions and partial domain monotonicity for the
  {N}eumann heat kernel.
\newblock {\em J. Funct. Anal.}, 86(2):226--236, 1989.

\bibitem{MS67}
H.~P. McKean, Jr. and I.~M. Singer.
\newblock Curvature and the eigenvalues of the {L}aplacian.
\newblock {\em J. Differential Geometry}, 1(1):43--69, 1967.

\bibitem{NRS19}
M.~Nursultanov, J.~Rowlett, and D.~A. Sher.
\newblock The heat kernel on curvilinear polygonal domains in surfaces,
  ar{X}iv:1905.00259 [math.{AP}], 2019.

\bibitem{NRS20}
M.~Nursultanov, J.~Rowlett, and D.~A. Sher.
\newblock How to hear the corners of a drum.
\newblock In {\em 2017 {MATRIX} Annals}, pages 243--278. Springer International
  Publishing, 2019.

\end{thebibliography}

\end{document}